\documentclass[11pt,reqno]{amsart}
\pdfoutput=1
\usepackage{amssymb}
\usepackage{graphicx}

\catcode`\@=11

\long\def\@savemarbox#1#2{\global\setbox#1\vtop{\hsize\marginparwidth 
  \@parboxrestore\tiny\raggedright #2}}
\newcommand\lref[1]{\ref{#1}%
\@ifundefined{r@DisplaY #1}{}{ (#1)}}

\newcommand\fakelabel[2]{\@bsphack\if@filesw {\let\thepage\relax
   \newcommand\protect{\noexpand\noexpand\noexpand}%
\xdef\@gtempa{\write\@auxout{\string
      \newlabel{#1}{{#2}{\thepage}}}}}\@gtempa
   \if@nobreak \ifvmode\nobreak\fi\fi\fi\@esphack}

\def\SL@margintext#1{{\showlabelsetlabel{\tiny\{\SL@prlabelname{#1}\}}}}
\catcode`\@=12

\def\Empty{}
\newcommand\oplabel[1]{
  \def\OpArg{#1} \ifx \OpArg\Empty {} \else
        \label{#1}
  \fi}
\newtheorem{theoremSt}{Theorem}[section]

\newtheorem{exampleSt}[theoremSt]{Example}
\newtheorem{exerciseSt}[theoremSt]{Exercise}

\newcommand\MakeStEnv[1]{
  \newenvironment{#1}[1]{
  \begin{#1St} \oplabel{##1}%
  \global\def\CrntSt{\thetheoremSt}%
}{ 
  \end{#1St} }
  \newenvironment{#1+}[1]{
  \begin{#1St} \label{##1}%
  \label{DisplaY ##1}%
  \global\def\CrntSt{\thetheoremSt}%
  \def\Labl{##1}\ifx\Labl\Empty{} \else {\em (\Labl)\,}\fi%
}{ 
  \end{#1St} }
}
\MakeStEnv{theorem}
\MakeStEnv{corollary}
\MakeStEnv{proposition}
\MakeStEnv{lemma}
\MakeStEnv{definition}
\MakeStEnv{conjecture}
\MakeStEnv{problem}
\MakeStEnv{question}

\long\def\state#1#2{
\medskip\par\noindent
{\bf #1} 
{\it #2}
\par\medskip
}

\long\def\realfig#1#2{
\begin{figure}[htbp]
\includegraphics{#1}
\caption[#1]{#2}
\oplabel{#1}
\end{figure}}

\newlength{\saveu}

\newenvironment{pf*}[1]{%
 \begin{proof}[#1]%
}{ 
 \end{proof}
}

\newcommand{\finishproof}[1]{ 
  \def\FPArg{#1}
  \ifx\FPArg\Empty
        \newcommand\FPArg{\CrntSt}  \fi
  \smallbreak\noindent\makebox[\textwidth]{\hfill\fbox{\FPArg}}
  \medbreak\noindent
}

\newcommand\CC{{\mathcal C}}

\newcommand\FF{{\mathcal F}}

\newcommand\LL{{\mathcal L}}
\newcommand\MM{{\mathcal M}}

\newcommand\PP{{\mathcal P}}

\newcommand\TT{{\mathcal T}}
\newcommand\UU{{\mathcal U}}
\newcommand\VV{{\mathcal V}}
\newcommand\WW{{\mathcal W}}

\newcommand\PMF{{\PP\kern-2pt\MM\FF}}

\newcommand\PML{{\PP\kern-2pt\MM\LL}}

\newcommand\half{{\textstyle{\frac12}}}

\newcommand\ep{\epsilon}

\newcommand\hhat{\widehat}

\newcommand\union{\cup}
\newcommand\intersect{\cap}
\newcommand\bbR{{\mathord{\text{I\kern-2pt R}}}}        
\newcommand\bbH{{\mathord{\text{I\kern-2pt H}}}}        

\newcommand\Z{{\mathbb Z}}
\newcommand\R{{\mathbb R}}

\newcommand\Hyp{{\mathbb H}}

\newcommand\bigrightarrow[1]{\hbox to #1{\rightarrowfill}}
\newcommand\bigleftarrow[1]{\hbox to #1{\leftarrowfill}}

\newcommand\boundary{\partial}
\newcommand\semidir{\mathrel{\hbox{\vrule depth-.03ex height1.1ex\kern-0.15em$\times$}}}

\newcommand\del{\nabla}

\newcommand{\diam}{\operatorname{diam}}

\numberwithin{equation}{section}

\catcode`\@=11

\def\subsection{\@startsection{subsection}{2}%
  \z@{.5\linespacing\@plus.7\linespacing}{.5em}%
  {\normalfont\bfseries\centering}}

\def\section{\@startsection{section}{1}%
  \z@{.7\linespacing\@plus\linespacing}{.5\linespacing}%
  {\normalfont\large\bfseries\centering}}

\def\subsubsection{\@startsection{subsubsection}{3}%
  \z@{.5\linespacing\@plus.7\linespacing}{-.5em}%
  {\normalfont\bfseries}}

\catcode`\@=12

\newcommand{\collar}{\operatorname{\mathbf{collar}}}

\newcommand{\fsubd}{\mathrel{{\scriptstyle\searrow}\kern-1ex^d\kern0.5ex}}
\newcommand{\bsubd}{\mathrel{{\scriptstyle\swarrow}\kern-1.6ex^d\kern0.8ex}}
\newcommand{\fsubeq}{\mathrel{\raise-.7ex\hbox{$\overset{\searrow}{=}$}}}
\newcommand{\bsubeq}{\mathrel{\raise-.7ex\hbox{$\overset{\swarrow}{=}$}}}
\newcommand{\tw}{\operatorname{tw}}

\newcommand{\bbar}{\overline}

\newcommand{\EL}{\mathcal{EL}}

\newcommand{\tsh}[1]{\left\{\kern-.9ex\left\{#1\right\}\kern-.9ex\right\}}

\newcommand\Teich{{\mathcal T}}

\newcommand\interior{{\rm int}}

\newcommand\qsim[1]{\stackrel{\scriptscriptstyle{#1}}{\sim}}

\newcommand{\bt}{{\mathbf t}}
\newcommand{\bl}{{\boldsymbol \lambda}}
\newcommand{\ba}{{\boldsymbol \alpha}}
\newcommand{\bb}{{\boldsymbol \beta}}
\newcommand{\bc}{{\bf c}}

\newcommand{\bU}{{\bf U}}
\newcommand{\bV}{{\bf V}}
\newcommand{\bW}{{\bf W}}
\newcommand{\bC}{{\bf C}}
\newcommand{\bD}{{\bf D}}
\newcommand{\cP}{{\mathcal P}}
\newcommand{\id}{{\operatorname{id}}}

\begin{document}

\title[Local topology in deformation spaces]{Local topology in deformation spaces of hyperbolic 3-manifolds}

\author[J. Brock]{Jeffrey F. Brock}
\address{Brown University}
\author[K. Bromberg]{Kenneth W. Bromberg}
\address{University of Utah}
\author[R. Canary]{Richard D. Canary}
\address{University of Michigan}
\author[Y. Minsky]{Yair N. Minsky}
\address{Yale University}
\date{\today}
\thanks{Bromberg was partially supported by NSF grants DMS-0554569 and DMS-0504877,
Brock was partially supported by DMS-0553694 and DMS-0505442, Canary was partially supported by DMS-0504791 and DMS-0554239, and Minsky
was partially supported by DMS-0504019 and DMS-05504321}

\begin{abstract}
We prove that  the deformation space $AH(M)$ of marked hyperbolic 3-manifolds
homotopy equivalent to a fixed compact 3-manifold $M$  with
incompressible boundary is locally connected at  minimally parabolic points.
Moreover, spaces of Kleinian surface groups are locally connected
at quasiconformally rigid points.
Similar results are obtained for deformation spaces of acylindrical 3-manifolds
and Bers slices.
\end{abstract}

\maketitle

\setcounter{tocdepth}{1}
\tableofcontents

\newcommand\epzero{\ep_0}
\newcommand\epone{\ep_1}
\newcommand\epotal{\ep_{\rm u}}
\newcommand\kotal{k_{\rm u}}
\newcommand\Kmodel{K_0}
\newcommand\Kone{K_1}
\newcommand\Ktwo{K_2}
\newcommand\bdry{\partial} 
\newcommand\stab{\operatorname{stab}}
\newcommand\nslices[2]{#2|_{#1}}
\newcommand\ME{M\kern-4pt E}
\newcommand\bME{\overline{M\kern-4pt E}}
\renewcommand\del{\partial}
\newcommand\s{{\mathbf s}}
\newcommand\pp{{\mathbf p}}
\newcommand\qq{{\mathbf q}}
\newcommand\uu{{\mathbf u}}
\newcommand\vv{{\mathbf v}}
\newcommand\zero{{\mathbf 0}}
\newcommand{\cB}{{\mathcal B}}
\newcommand\mm{\operatorname{\mathbf m}}
\newcommand{\dehntw}{\theta}

\section{Introduction}

The conjectural picture for the topology of the deformation space
$AH(M)$ of all (marked) hyperbolic 3-manifolds homotopy equivalent to
a fixed compact 3-manifold $M$ has evolved from one of relative
simplicity to one far more complicated in recent years.  Indeed, the
interior of this space has been well-understood since the late
1970's. Roughly, components of $AH(M)$ are enumerated by (marked)
homeomorphism types of compact 3-manifolds homotopy equivalent to $M$,
and each component is a manifold parameterized by natural conformal
data.  In the last decade, however, a string of results has
established that the topology of $AH(M)$ itself is not well-behaved.
In particular, $AH(M)$ fails to be locally connected when $M$ is an
untwisted $I$-bundle over a closed surface (\cite{bromberg-PT},
\cite{Magid}), and a new conjectural picture in which such pathology
is prevalent has replaced the old.

The present paper clarifies the role that the geometry and topology of
3-manifolds associated to points in the boundary of $AH(M)$ plays in
the local topology at such points.  In particular, we show that the
topology of $AH(M)$ is well-behaved at many points; if $M$ has
incompressible boundary, then $AH(M)$ is locally connected at
``generic'' points in the boundary.  When $M$ is acylindrical or an
untwisted $I$-bundle we obtain finer results.  

Central to the present discussion are recent fundamental improvements
in the understanding of the internal geometry and topology of ends of
hyperbolic 3-manifolds.  Via the Ending Lamination Theorem of
\cite{ELC1,ELC2} and the model manifold developed in its proof, the
Tameness Theorem of \cite{agol,calegari-gabai} and the Density Theorem
\cite{ELC2,namazi-souto,ohshika-density}, we develop a more complete
picture of the topological complexity at the boundary of deformation
spaces.

\bigskip

Our first theorem extracts consequences for the local structure of
deformation spaces in terms of the topology of $M$ and the presence of
parabolic elements for an element $\rho$ in the boundary of $AH(M)$.

Two components $B$ and $C$ of ${\rm int}(AH(M))$ are said to
{\em  bump} at $\rho\in \partial AH(M)$ if $\rho\in \overline{B}\cap
\overline{C}$.  A component $B$ of ${\rm int}(AH(M))$ is said to
{\em  self-bump} at $\rho\in \partial B$ if there exists a neighborhood
$W$ of $\rho$ such that if $V$ is a neighborhood of $\rho$ which is
contained in $W$, then $V\cap B$ is disconnected. A point
$\rho\in \partial AH(M)$ is said to be {\em uniquely approachable} if
there is no bumping or self-bumping at $\rho$. 
The Density Theorem \cite{ELC2,namazi-souto,ohshika-density} 
asserts that
$AH(M)$ is the closure of its interior, so $AH(M)$ is
locally connected at all uniquely approachable points.

\begin{theorem}{nocuspscase}{}
  Let $M$ be a compact 3-manifold with incompressible boundary and
  $\rho\in \partial AH(M)$.  If every parabolic element of
  $\rho(\pi_1(M))$ lies in a rank-two free abelian subgroup, then
  $\rho$ is uniquely approachable. In particular, $AH(M)$ is locally
  connected at $\rho$.
\end{theorem}
\noindent {\bf Remark:}
Such points $\rho$ are generic in the
boundary of $AH(M)$ in the sense of Lemma
4.2 in \cite{canary-hersonsky}. \smallskip

Recall that if $\rho\in AH(M)$, then
$N_\rho=\mathbb{H}^3/\rho(\pi_1(M))$ is a hyperbolic 3-manifold
homotopy equivalent
to $M$. If $\Omega(\rho)$ is the domain of discontinuity for the
action of $\rho(\pi_1(M))$ on $\widehat{\bf C}$, then
$\partial_cN_\rho=\Omega(\rho)/\rho(\pi_1(M))$ is a Riemann surface
called the {\em conformal boundary} of $N_\rho$.  In order to rule out
bumping in the presence of parabolics we place the additional
restriction on $\rho$ that every component of its conformal boundary
is a thrice-punctured sphere.  Such a $\rho$ is called {\em
  quasiconformally rigid}.  Notice that this includes the case that
the conformal boundary is empty.

\begin{theorem}{nobumping}{}
Let $M$ be a compact 3-manifold. If $\rho$ is a quasiconformally
rigid point in $\partial AH(M)$, then there is no bumping at $\rho$.
\end{theorem}

In order to rule out self-bumping, we make additional
restrictions on the topology of $M$.

\begin{theorem}{noselfbump}{}
 Let $M$ be a compact 3-manifold which is either acylindrical or
  homeomorphic to $S\times I$, for a closed surface $S$. If $\rho$ is
  a quasiconformally rigid point in $\partial AH(M)$ then there is no
  self-bumping at $\rho$.
\end{theorem}

We may combine Theorems \ref{nobumping} and \ref{noselfbump} to
establish the following corollary.

\begin{corollary}{}{}
  Let $M$ be a compact 3-manifold which is either acylindrical or
  homeomorphic to $S\times I$, for a closed surface $S$. If $\rho$ is
  a quasiconformally rigid point in $\partial AH(M)$ then $\rho$ is
  uniquely approachable. In particular, $AH(M)$ is locally connected
  at $\rho$.
\end{corollary}

If $M = S\times I$, then ${\rm int}(AH(S\times I))$ is the 
{\em  quasi-Fuchsian locus}, denoted $QF(S)$, and is naturally identified
with $\mathcal{T}(S)\times \mathcal{T}(S)$.  Given 
$Y\in \mathcal{T}(S)$, the {\em Bers slice} $B_Y$ of
$QF(S)$ is the slice $\mathcal{T}(S)\times \{Y\}$ in the
product structure.  If $\rho$ lies in the boundary of a Bers slice
$B$, then its conformal boundary always has a component homeomorphic
to $S$ (see Bers \cite[Theorem 8]{bers-slice}).  In this setting, we
say that $\rho$ is {\em quasiconformally rigid in $\bdry B$} if every other
component of its conformal boundary is a thrice-punctured sphere. We
say a Bers slice {\em self-bumps} at a point $\rho\in\partial B$ if there
exists a neighborhood $W$ of $\rho$ in the closure $\overline B$ of
$B$ (within $AH(S\times I)$) such that if $V$ is a neighborhood of
$\rho$ in $\overline B$ which is contained in $W$, then $V\cap B$ is
disconnected.

\begin{theorem}{bersslice}{}
  Let $B$ be a Bers slice of $QF(S)$ for some closed surface $S$.  If
  $\rho\in \partial B$ and $\rho$ is quasiconformally rigid in
  $\partial B$, then $B$ does not self-bump at $\rho$. In particular,
  its closure $\bar B$ is locally connected at $\rho$.
\end{theorem}

\noindent {\bf History.}
The Ending Lamination Theorem \cite{ELC1,ELC2,ELC3} asserts that 
hyperbolic 3-manifolds in $AH(M)$ are classified
by their (marked) homeomorphism type
and ending invariants which encode the asymptotic
geometry of their ends.  As points in the interior are
parametrized by Teichm\"uller space(s) and ending laminations are
associated to points on the boundary, a tenuous analogy between
deformation spaces and Thurston's compactification of Teichm\"uller
spaces by the sphere of projective measured laminations clouded the
picture of the topological structure of deformation spaces for many
years.  The non-continuity of the action of the mapping class group on
Bers compactification \cite{kerckhoff-thurston}, illustrated some
initial failings of this analogy, and elucidated a central example of
J\o rgenson (see \cite{marden-bulletin}) concerning the disparity between algebraic
and geometric convergence that underlies the present discussion.

Anderson and Canary \cite{ACpages} showed that the (marked)
homeomorphism type need not vary continuously over $AH(M)$, while
Brock \cite{brock-invariants} showed that ending laminations do not
vary continuously in any of the usual topologies, even in the closure
of a Bers slice. These results make it clear that the parameterization
of $AH(M)$ must be much more complicated than one might naively hope.

Bumping phenomena in deformation spaces were first discovered by
Anderson and Canary \cite{ACpages}. Anderson, Canary and McCullough
\cite{ACM} characterized exactly which components of 
${\rm  int}(AH(M))$ bump when $M$ has incompressible boundary.  McMullen
\cite{mcmullenCE} showed that $QF(S)$ self-bumps,
while Bromberg and Holt \cite{bromberg-holt}
showed that every component of ${\rm int}(AH(M))$ self-bumps whenever
$M$ contains a primitive essential annulus. For a more complete overview
of recent results on the pathology of the topology of $AH(M)$ see
\cite{canary-bumponomics}.

All known bumping and self-bumping results make use of the
``wrapping'' construction from \cite{ACpages} which requires the
presence of a primitive essential annulus. It is not yet known whether
self-bumping can occur in $AH(M)$ when $M$ does not contain primitive
essential annuli or in the closure of a Bers slice.  However, Bromberg
\cite{bromberg-PT} conjectures that if $S$ is a closed surface of
genus at least 2, then the closure of every Bers Slice of $QF(S)$ is
not locally connected.  In the case of Bers slices of the space of
punctured torus groups, Minsky \cite{minskyPT} showed that the closure
of every Bers slice is a disk and hence locally connected. We
conjecture, similarly, that $AH(M)$ is not locally connected whenever
$M$ has a boundary component of genus at least two.

\subsection*{Outline of the Argument}

In section 3, we rule out bumping in the setting of Theorems
\ref{nocuspscase} and \ref{nobumping}. In each case, the point is to
rule out change of marked homeomorphism type in a sequence approaching
the point in question.  The hypotheses allow for the key use of the
core embedding results of Anderson-Canary-Culler-Shalen \cite{ACCS}.

In section 4, we rule out self-bumping in the setting of Theorem
\ref{nocuspscase}.  By hypothesis, we consider a point $\rho$ with no
extra parabolics and some degenerate ends. To rule out self-bumping at
$\rho$ it suffices to consider two sequences $\{\rho_n\}$ and
$\{\rho'_n\}$ in $\interior(AH(M))$ converging to $\rho$, and show that they
can be connected by a sequence of paths $\{\gamma_n\}$, also in
$\interior(AH(M))$, which accumulate only on $\rho$. Non-bumping implies
that $\rho_n$ and $\rho'_n$ are quasiconformally conjugate, so the
paths can be chosen as Teichm\"uller geodesics in the associated
quasiconformal deformation space. We can control the behavior of the
ending invariants of these sequences, and use the Ending Lamination
Theorem to show that any accumulation point of these paths is $\rho$.

The proof of Theorem \ref{bersslice} (the Bers slice case) is given in
Section \ref{bers slice}, using results from Sections \ref{section: FN projections}
and \ref{deformations}.  For clarity, consider first the case of a
point $\rho\in\partial B$ which is a {\em maximal cusp}; that is,
where a maximal curve system $\ba$ on the base surface $S$ is
represented by parabolics.

There is a neighborhood basis of $\rho$ in $B$ consisting of sets of
the form
$$
U(\delta)  = \{\rho'\in B: l_{\alpha_j}(\rho') < \delta \quad \forall \alpha_j\in\ba\}
$$
where $l_{\alpha_j}(\rho')$ is the translation distance in hyperbolic
space of $\rho'(\alpha_j)$, for a component $\alpha_j$ of $\ba$.  To
show no self-bumping occurs at $\rho$, then, we must show that for any
$\ep>0$ there is a $\delta>0$ such that any two points in $U(\delta)$
can be joined by a path in $U(\ep)$.

To show this would be straightforward if all components of $\ba$ were
already short on the top conformal boundary of our group:
Fenchel-Nielsen coordinates for the Teichm\"uller space of the varying
conformal boundary component in the Bers slice can be used directly to
obtain a path in which the lengths of components of $\ba$ are
controlled.

In general, however, curves in $\ba$ can have very short geodesic
representatives deep inside the convex core of the manifold, while on
the boundary they are extremely long.  To obtain geometric control
over the interior of the convex core via boundary geometry requires
tools from the solution of the Ending Lamination Conjecture in
\cite{ELC1} and \cite{ELC2}.  Lemma \ref{twist and shrink} gives the
final statement needed, namely that when the geodesic representatives of
$\ba$ are  very short in the manifold, there is a continuous path in $B$
terminating at a point where $\ba$ is short in the
conformal boundary and the geodesic representatives of $\ba$ are short
in all the corresponding hyperbolic manifolds along the deformation.

We develop the necessary machinery for finding such a path in Sections
\ref{section: FN projections} and \ref{deformations}.  Recall first
from \cite{ELC1} that short length for a curve $\gamma$ in a surface
group corresponds to {\em large projection coefficients} for some
subsurface $W$ with $\gamma\subset \boundary W$.  That is, for each
subsurface $W$ we project the ending invariants of the group to the
curve complex $\CC(W)$ and measure the distance between them.  Then a curve
$\gamma \in \CC(S)$ is short in the hyperbolic 3-manifold if and only
if it is short in the conformal 
boundary or one of these coefficients is large for a subsurface with
$\gamma$ in its boundary (see Theorem \ref{length and projections} for
a precise statement).

In section \ref{section: FN projections} we examine Fenchel-Nielsen
coordinates and their effect on subsurface projections. In particular
we prove in Theorem \ref{FN projections} that, given a curve system
$\ba$ and a point $X$ in Teichm\"uller space, we can deform the length
and twist parameters of $X$ as much as we want without changing by
more than a bounded amount projections to subsurfaces disjoint from
$\alpha$.

In Section \ref{deformations} we perform the deformation. The
trickiest issue is that we must adjust the components of the curve
system in an order reflecting their arrangement in the manifold, with
the curves ``closest'' to the top boundary being adjusted first. In
particular, to each component $\alpha_i$ of $\ba$ we associate a
subsurface $W_i$ with $\alpha_i$ in its boundary, whose projection
coefficient is large enough to be responsible for $\alpha_i$ being
short. To each $W_i$ is associated a certain geometric region in the
manifold, and these regions are partially ordered in terms of their
separation properties in the manifold. In order not to disturb the
projection coefficients of the other surfaces while adjusting each
$\alpha_i$, we need to start with the highest ones.

In practice we detect this partial order in a combinatorial way, by
looking at the projections of the subsurface boundaries to each
other's curve complexes. These ideas come from \cite{masur-minsky} and
\cite{ELC2}, and are also exploited in \cite{BKMM} and elsewhere.  The
details of this are discussed in Section \ref{partial orders}, and
Lemma \ref{partialorder} in particular.

In the general case of Theorem \ref{bersslice}, we must consider a
representation $\rho$ with a mix of parabolics (a non-maximal system $\ba$)
and degenerate ends. By the Ending Lamination Theorem such representations are
uniquely determined by their ending invariants, and we can determine a
neighborhood system for $\rho$ by considering constraints not just on the
lengths of the curves in $\ba$ but on the projections of the ending data to
the subsurfaces associated to the degenerate ends. The appropriate statement is
given in Lemma \ref{U is nbhd}, which relies on a Theorem
\ref{endlams},  whose proof will appear in \cite{pull-out}. 

The acylindrical case of Theorem \ref{noselfbump} is handled in
Section \ref{acylindrical}. This is quite similar to the Bers slice
case, with Thurston's Bounded Image Theorem providing control on the
lower conformal boundary of each boundary subgroup.

Finally, the general surface group case of Theorem \ref{noselfbump} is
completed in Section \ref{general surface}.  In this case parabolics
and degenerate ends can occur on both top and bottom.  We deform one
end and then the other, taking care to preserve the order of the ends
(and, in particular, the order of the curves becoming parabolic).

\medskip
\noindent {\bf Acknowledgements.}  The authors gratefully acknowledge
the support of the National Science Foundation and the support of
their NSF FRG grant in particular.  We also thank Francis Bonahon for suggesting an
initial form of the argument for Theorem~\ref{nocuspscase}.

\section{Background}
\label{background}

In this section, we recall some of the key tools and results which
will be used in the paper. (A few new technical lemmas will be derived
in sections \ref{partial orders} and \ref{geometric limits}).

In section \ref{ELT} we survey the Ending Lamination Theorem
which provides a classification of hyperbolic 3-manifolds
with finitely generated fundamental group in terms of their
ending invariants. In section \ref{AHback}, we recall basic facts about
deformation spaces of hyperbolic 3-manifolds, for example the parameterization
of the interior of $AH(M)$ and Thurston's Bounded Image Theorem.
In section \ref{ending invariants and geometry}, we recall results
which explain how the internal geometry of hyperbolic 3-manifolds
can be detected from its ending invariants, via subsurface projections.
In section \ref{partial orders}, we introduce the partial order on
(certain) subsurfaces discussed in the outline of argument and relate
it to the ordering of curves in the hyperbolic 3-manifold. In section
\ref{geometric limits}, we recall basic facts about geometric limits
and derive consequences of the core embedding results of \cite{ACCS}.

\subsection{Ending invariants and the Ending Lamination Theorem}
\label{ELT}

We recall (see \cite{BP} for example)
that there exists a Margulis constant $\mu>0$, such that if
$\epsilon<\mu$ and 
$$N_{{\rm thin}(\epsilon)}=\{x\in N\ |\ {\rm  inj}_N(x)<\epsilon\},$$
then every component of $N_{{\rm thin}(\epsilon)}$ is either a solid torus,
metric neighborhood of a closed geodesic in $N$ or a ``cusp'', i.e. a quotient
of a horoball in ${\bf H}^3$ by a group of parabolic transformations, which
is homeomorphic to $T\times (0,\infty)$ where $T$ is either a torus or
an open annulus. We pick a uniform $\epsilon_0<\mu$ which will
be used throughout the paper.

If $\rho\in AH(M)$, let $N_\rho={\bf H}^3/\rho(\pi_1(M))$ and let
$N_\rho^0$ be obtained from $N_\rho$ by removing all the cusps
of $(N_\rho)_{{\rm thin}(\epsilon_0)}$. A {\em compact core} for a hyperbolic
3-manifold $N$ is a compact submanifold $C$ such that the inclusion of $C$
into $N$ is a homotopy equivalence. A {\em relative compact core} $M_\rho$ for
$N_\rho$ is  a compact core for $N_\rho^0$ which intersects every component
of $\partial N_\rho^0$ in a compact core for that component.
(The existence of a relative compact core is due to Kulkarni-Shalen
\cite{kulkarni-shalen} and McCullough \cite{mcculloughRCC}.)
Let $P_\rho=M_\rho\cap\partial N_\rho^0$.
There exists a well-defined, up to homotopy,
homotopy equivalence $h_\rho:M\to M_\rho$ in the
homotopy class determined by $\rho$, and a well-defined identification of the
conformal boundary $\partial_cN_\rho$ 
with a collection of components of $\partial M_\rho-P_\rho$.
The Tameness Theorem  of Agol \cite{agol} and 
Calegari-Gabai \cite{calegari-gabai} assures us that we may choose $M_\rho$
so that $N_\rho^0-M_\rho$ is homeomorphic to 
$(\partial M_\rho-P_\rho)\times (0,\infty)$.

If a component $S$ of $\partial M_\rho-P_\rho$ is identified with a component of
$\partial_cN_\rho$, it is called {\em geometrically finite} and inherits a natural conformal structure, regarded as a point in  $\mathcal{T}(S)$.
Otherwise, the component $S$ is called {\em geometrically infinite} and it 
bounds a neighborhood of a geometrically infinite end. There exists a collection
of simple closed curves $\{\alpha_i\}$ on $S$, whose geodesic representatives
lie in the component of $N_\rho^0-M_\rho$ bounded by $S$ and leave every compact
set. Regarded as a sequence of projective measured laminations, $\{\alpha_i\}$ converges to
$\mu\in PL(S)$. The support $\lambda$ of $\mu$, regarded as
a geodesic lamination, is called the {\em ending lamination} associated to $S$.
The ending lamination $\lambda$ lies in the set $\EL(S)$ of
geodesic laminations admitting measures of full support which {\em
  fill the surface:}  every component of their complement is a disk or
a peripheral annulus.
(See Thurston \cite{thurston-notes},   Bonahon \cite{bonahon-preprint} and
Canary \cite{canary-ends} for a discussion of geometrically infinite ends
and their ending laminations).
The Ending Lamination Theorem (see Minsky \cite{ELC1} 
and Brock-Canary-Minsky \cite{ELC2,ELC3}) tells us
that this information determines the manifold up to isometry.

\state{Ending Lamination Theorem}{
Suppose that $\rho_1,\rho_2\in AH(M)$, then
$\rho_1=\rho_2$ if and only if there exists an orientation-preserving homeomorphism of pairs
$g:(M_{\rho_1},P_{\rho_1})\to (M_{\rho_2},P_{\rho_2})$ such that
\begin{enumerate}
\item
$g\circ h_{\rho_1}$ is homotopic to $h_{\rho_2}$,
\item
$g$ is a conformal homeomorphism from the geometrically finite components
of $\partial M_{\rho_1}-P_{\rho_1}$ to the geometrically finite components of
$\partial M_{\rho_2}-P_{\rho_2}$, and
\item $g$ takes the ending lamination of any geometrically infinite component of
$\partial M_{\rho_1}-P_{\rho_1}$ to the ending lamination of the image
geometrically infinite component of
$\partial M_{\rho_2}-P_{\rho_2}$.
\end{enumerate}
}

\subsection{Deformation spaces of hyperbolic 3-manifolds}
\label{AHback}

We begin by reviewing the classical deformation theory of the interior
of $AH(M)$. (See section 7  of Canary-McCullough \cite{canary-mccullough} for
a complete treatment of this theory and its history.)
Let $\mathcal{A}(M)$ denote the set of (marked) homeomorphism types of
compact, oriented hyperbolisable 3-manifolds homotopy equivalent to $M$. We recall
that $\mathcal{A}(M)$ is the set of pairs $(M',h')$ where $M'$ is an oriented,
hyperbolisable compact 3-manifold and $h:M\to M'$ is a homotopy equivalence,
where $(M_1,h_1)$ and $(M_2,h_2)$ are said to be equivalent if there exists
an orientation-preserving homeomorphism $j:M_1\to M_2$ such that
$j\circ h_1$ is homotopic to $h_2$. We get a well-defined map 
$$\Theta:AH(M)\to \mathcal{A}(M)$$
given by taking $\rho$ to the equivalence class of $(M_\rho,h_\rho)$.
This map is surjective and the components of the interior of $AH(M)$ are exactly
the pre-images of points in $\mathcal{A}(M)$. 

If $M$ has incompressible boundary, equivalently if $\pi_1(M)$ is freely indecomposable,
then  points in $\Theta^{-1}(M',h')\cap {\rm int}(AH(M))$ give rise
to well-defined conformal structures on $\partial_TM'$, where $\partial_TM'$ is the
set of non-toroidal boundary components of $\partial M'$. Moreover, every possible
conformal structure arises and the conformal structure determines the manifold.
Therefore, we may identify the component $\Theta^{-1}(M',h')\cap {\rm int}(AH(M))$
with $\mathcal{T}(\partial_TM')$.

The Density Theorem asserts that $AH(M)$ is the closure of its interior.
If $M$ has incompressible boundary, the Density Theorem follows from
the Ending Lamination Theorem \cite{ELC1,ELC2},
Bonahon's Tameness Theorem \cite{bonahon}
and convergence results of Thurston \cite{thurston2,thurston3}.  (A discussion
of the history of this proof in the general case is contained in
\cite{canary-bumponomics}.)
There is an alternate approach, using cone-manifold deformation theory, pioneered
by Bromberg \cite{bromberg-density} and Brock-Bromberg \cite{brock-bromberg}
and completed by Bromberg-Souto \cite{bromberg-souto}. 

The majority of this paper will be concerned with the case where $M=S\times I$
and $S$ is a closed surface. In this case,
$\mathcal{A}(S\times I)$ is a single point, and the interior $QF(S)$  of $AH(S\times I)$
(which is often abbreviated to $AH(S)$) is identified with
$\mathcal{T}(S)\times\mathcal{T}(S)$. If $\rho\in AH(S)$, then $M_\rho$ is identified
with $S\times [0,1]$. (Here we are implicitly identifying
$\mathcal{T}(\overline{S})$ with $\mathcal{T}(S)$ where $\overline{S}$ is $S$
with the opposite orientation. Formally, the conformal structure on $S\times \{ 0\}$
lies in $\mathcal{T}(\overline{S})$.)
The orientation on $S$ allows us to identify one component $\partial_1M_\rho$
as the top, or {\em upward pointing} component and the other component $\partial_0M_\rho$ as the bottom
or {\em downward pointing} component. If $\rho\in QF(S)$ has conformal structure
$X$ on $\partial_1M_\rho$ and $Y$ on $\partial_0M_\rho$,
we will use the notation $\rho=Q(X,Y)$.
In general, $P\cap \partial_1M_\rho$ may be
identified with the regular neighborhood of a collection $\ba$
of simple closed curves on $S$ and  $P\cap \partial_0M_\rho$ may also be identified with the regular neighborhood of a collection
$\bb$ of simple closed curves on $S$. We say that the components of $\ba$ are
associated to upward-pointing cusps, while the components of $\bb$
are associated to downward-pointing cusps.
Similarly the components of $\partial_1M_\rho \smallsetminus P$ are
said to  bound upward-pointing ends, and the components of
$\partial_0M_\rho \setminus P$ are said to bound downward-pointing ends.
If $\rho\in AH(S)$ is quasiconformally rigid, a
component of $\partial M_\rho-P_\rho$ is geometrically finite if and only if it is a thrice-punctured sphere,
while the remaining components each bound neighborhoods of degenerate ends and inherit an ending lamination.

\medskip

We recall that a Bers slice $B_Y$ of $QF(S)$ is a set of the form
$\mathcal{T}(S)\times \{ Y\}$ where $Y\in\mathcal{T}(S)$.
If $B_Y$ is a Bers slice and $\rho\in \overline{B_Y}$ (the closure of $B_Y$ in
$AH(S)$), then the bottom boundary
component of $M_\rho$ is geometrically finite and has conformal structure $Y$
(see Bers \cite[Theorem 8]{bers-slice}).
If $\rho$ is quasiconformally rigid in $B_Y$,
one then obtains a collection $\ba$ of curves on the
top boundary component whose regular neighborhood is $P_\rho$,
and an ending lamination on every upward-pointing
component of $\partial M_\rho-P_\rho$ which is not a thrice-punctured sphere.

\medskip

The other special  case  we will consider is when $M$ is acylindrical.
Johannson \cite{johannson} showed that any homotopy equivalence
from an acylindrical manifold to a compact 3-manifold is homotopic
to a homeomorphism, so $\mathcal{A}(M)$ has two component
(one associated to each possible orientation on $M$).
So, ${\rm int}(AH(M))$ has  two components and it follows from \cite{ACM}
that $\Theta$ is locally constant. Thurston \cite{thurston1} showed
that $AH(M)$ is compact if $M$ is acylindrical. 

Our proof of Theorem \ref{noselfbump} in the acylindrical case we will
make crucial use of Thurston's Bounded Image Theorem
(see Kent \cite{kent} for a proof.) If $B$ is a component of ${\rm int}(AH(M))$
then $B$ is identified with $\mathcal{T}(\partial_T M)$. If $S$ is a component
of $\partial_T M$, then there is a natural map $r_S:B\to AH(S)$ given
by restriction, whose image lies in $QF(S)$. If $\tau\in\mathcal{T}(\partial M)$,
then $r_S(\tau)$ is a well-defined point
$(\tau|_S,\sigma_S(\tau))$ where $\sigma_S(\tau)\in\mathcal{T}(S)$. Letting
$S$ vary over all components of $\partial_T M$, we get a well-defined
map 
$$\sigma:\mathcal{T}(\partial_T M)\to \mathcal{T}(\partial_T M)$$
called the skinning map. Thurston's Bounded Image Theorem
simply asserts that $\sigma$ has bounded image in $\mathcal{T}(\partial M)$.

\subsection{The conformal boundary of a hyperbolic 3-manifold and its
internal geometry}
\label{ending invariants and geometry}

In this section, we review a variety of results which relate the geometry
of the conformal boundary to the geometry of the hyperbolic 3-manifold.
Most classically, a result of Bers \cite{bers-slice} shows that lengths of
curves in the conformal boundary provide upper bounds for lengths in
the manifold. To set notation, if $\rho\in AH(M)$ and $\alpha$ is a (homotopically
non-trivial) closed
curve in $M$, then $l_\rho(\alpha)$ is the length of the geodesic
representative $\alpha^*$ of $h_\rho(\alpha)$ in $N_\rho$ (with $l_\rho(\alpha)=0$ if 
$h_\rho(\alpha)$
is homotopic into a cusp of $N_\rho$). Similarly, if $X\in \mathcal{T}(S)$ and
$\alpha$ is a closed curve on $X$,
then $l_X(\alpha)$ is the length of the geodesic representative of $\alpha$
on $X$.

\begin{lemma}{berslemma}{\rm (Bers \cite[Theorem 3]{bers-slice})}
If $\rho=Q(X,Y)\in QF(S)$, then 
$$l_\rho(\alpha)\le 2l_X(\alpha)$$ for
any closed curve $\alpha$ on $X$. 
\end{lemma}

\subsubsection*{Subsurface projections and the curve complex}
The proof of the Ending Lamination Theorem develops
more sophisticated information about the relationship
between the geometry of a hyperbolic 3-manifold and its ending invariants.
This information is typically expressed in terms of projections onto
curve complexes of subsurfaces of the boundary.

Recall from \cite{masur-minsky2} the curve complexes $\CC(W)$ where
$W\subseteq S$ is an essential subsurface. When $W$ is not an annulus,
the vertices of $\CC(W)$ are homotopy classes of simple closed
nonperipheral curves in $W$. When $W$ is an annulus, vertices are
homotopy classes rel endpoints of arcs connecting the boundaries of
the compactified annulus cover $\hhat W\to S$ associated to $W$. 
Edges in these complexes correspond to pairs of vertices with
representatives that intersect in the minimal possible number of points
allowed by $W$. $\CC(W)$ is endowed with the path metric $d_{\CC(W)}$ assigning
length 1 to each edge. 
If $W$ is a three-holed sphere then $\CC(W)$ is empty, and from now on
we implicitly ignore this case. 

If $\CC(S,W)$ denotes the set of curves in $S$ which intersect $W$
essentially, we have, also as in \cite{masur-minsky2}, 
{\em subsurface  projection maps}
$$
\pi_W:\CC(S,W)\to\CC(W).
$$
If $W$ is not an annulus then $\pi_W(\alpha)$ is obtained by selecting
(any) arc of the essential intersection of $\alpha$ with $W$, and
doing surgery with $\boundary W$ to obtain a closed curve. When $W$ is
an annulus we take more care: we consider the
annular cover $\hat W$ of $S$ associated to $W$ and lift $\alpha$ to
an arc connecting the two boundaries. All the choices involved in
these constructions differ  by bounded distance in the image, and in
our applications this ambiguity will not matter. 
Define, for $\alpha,\beta\in\CC(S,W)$, 
$$
d_W(\alpha,\beta )  = d_{\CC(W)}(\pi_W(\alpha),\pi_W(\beta)).
$$

All of these notions can be applied to points in $\Teich(S)$ as well,
giving a map 
$$
\pi_W: \Teich(S) \to \CC(W)
$$
defined as follows: Given $X\in\TT(S)$ let $\alpha$ be a curve of
minimal length in $X$ intersecting $W$ essentially and let $\pi_W(X) =
\pi_W(\alpha)$. Except when $W$ is an annulus (or a three-holed
sphere, which we always exclude), the length of $\alpha$
has a uniform upper bound known as the Bers constant of $S$ (see \cite{bers-constant}).
Indeed the shortest maximal curve system has a uniform upper length bound,
and one of those curves must intersect $W$. 
Any non-uniqueness in the choice of $\alpha$ leads to values for
$\pi_W(X)$ that differ by a uniformly bounded amount. 

If $W$ is an annulus whose core $\gamma$ has extremely short length in
$X$, then the shortest curve crossing $\gamma$ will be 
long; however, the ambiguity in the definition of $\pi_W$ will still
be uniformly bounded. To see this, note that if two curves
$\beta_1$ and $\beta_2$ crossing
$\gamma$ have projections with distance greater than 2 in
$\CC(\gamma)$, then there exists a pair of arcs $b_1$ and $b_2$ in
$\beta_1$ and $\beta_2$ respectively with common endpoints whose
concatenation is homotopic into $\gamma$. Exchange of these arcs, and
smoothing,  will strictly shorten at least one of $\beta_1$ or
$\beta_2$, so they cannot both have minimal length in $X$. 
(The same argument actually works for non-annular $W$ as well).

\subsubsection*{Lengths in Kleinian surface groups}
In the case of a quasifuchsian hyperbolic manifold $Q(X,Y)$, a curve is short
if and only if it is either short in the conformal boundary or there is
a subsurface with the curve in its boundary such that $d_W(X,Y)$ is large.
To be more explicit,
given a simple closed curve $\gamma$ in $S$ and $X,Y\in\Teich(S)$, we
define 
$$
\mm_\gamma(X,Y) = 
\max\left(
\frac{1}{l_\gamma(X)},\frac{1}{l_\gamma(Y)},
\sup_{\gamma\subset\boundary W} d_W(X,Y)
\right).
$$
The supremum is over all essential subsurfaces in $S$ whose boundary
contains a curve parallel to $\gamma$.  The following theorem is a
re-statement (and special case) of the Length Bound Theorem from
Brock-Canary-Minsky \cite{ELC2}.

\begin{theorem}{length and projections}
Given $\ep>0$ there exists $M$ such that, for any $Q(X,Y)\in QF(S)$, and
simple closed curve $\gamma $ in $S$, 
$$
\mm_\gamma(X,Y) > M \implies l_\gamma(Q(X,Y)) < \ep.
$$
Conversely, given $M'$ there exists $\ep'>0$ such that
$$
 l_\gamma(Q(X,Y)) < \ep' \implies \mm_\gamma(X,Y) > M'.
$$
\end{theorem}

\subsection{Partial orders}
\label{partial orders}

In view of Theorem \ref{length and projections}, those subsurfaces $W$
where $d_W(X,Y)$ is large are important because their
boundaries correspond to short curves in $Q(X,Y)$. If the
curves are sufficiently short then
Otal \cite{otal-knotting} shows that their
associated Margulis tubes are {\em unlinked},
meaning they are isotopic to level curves in a product
structure on $Q(X,Y)$, and hence admit a natural partial order.

If $\alpha,\beta\in \mathcal{C}(S)$
and $i(\alpha,\beta)\ne 0$, then we say that $\alpha$ {\em lies above} $\beta$,
and that $\beta$ {\em lies below} $\alpha$, in $N_\rho\in AH(S)$ if
their geodesic  representatives $\alpha^*$ and $\beta^*$ are disjoint
and $\alpha^*$ may be homotoped to $+\infty$ in the complement of
$\beta^*$ (that is, there is a proper map $F:S^1\times [0,\infty)\to N_\rho$ such
that $F|_{S^1\times \{0\}}=\alpha^*$, $\beta^*\cap F(S^1\times [0,\infty))=\emptyset$,
and $F(S^1\times\{t\})$ is a family of curves exiting the upward-pointing end
of $N_\rho$). 
If $l_\rho(\alpha)=0$, then $\alpha$ lies above $\beta$ if $\alpha$ is associated
to an upward-pointing cusp.
See Section 3.1 of \cite{ELC2} for further discussion of
this topological partial order.

There is a closely related combinatorial partial order, 
which originates in the ``hierarchy path'' construction of \cite{masur-minsky2}.

For $(X,Y)$ an ordered pair of points in Teichm\"uller space
and $c>0$, define the following collection of (isotopy
classes of) essential
subsurfaces of $S$:
$$
\LL_c(X,Y) = \{ W\subset S: d_W(X,Y) > c \}.
$$
We say two subsurfaces or curves in $S$ {\em overlap} if they
intersect essentially and neither is contained in the other.

The lemma below can be extracted from Lemmas 4.18, 6.1 and 6.2 of
\cite{masur-minsky2} (see also section 4.1 of Behrstock-Kleiner-Minsky-Mosher
\cite{BKMM}). 

\begin{lemma}{partialorder}
There is a constant $m_1$ such that, if $c > m_1$ then
$\LL_c \equiv \LL_c(X,Y)$ admits a partial order $\prec$, such that 
any $U,V\in \LL_c$ which overlap are ordered, and $U\prec V$ implies
that 
\begin{enumerate}
\item
$ d_U(\boundary V,X)\le m_1$,
\item 
$d_U(\boundary V,Y) > c-m_1,$
\item
$  d_V(Y,\boundary U) \le m_1,$ and
\item
$d_V(\boundary U,X) > c-m_1.$
\end{enumerate}
Moreover, if $U\in\LL_c(X,Y)$,  $c>2m_1$, $V$ overlaps $U$, and
$
d_V(X,\boundary U) > m_1,
$
then
$$(5)\ \ d_V(X,Y) > c-m_1$$
and 
$$(6)\  \ U\prec V $$
with respect to the order on $\LL_{c-m_1}(X,Y)$. 
\end{lemma}
One way to make sense of these inequalities is to interpret a large
value for $d_U(\boundary V,X)$ to mean that $U$ is ``between'' $V$ and
$X$. In \cite{masur-minsky2} this had a literal meaning, because a
large value for $d_U(\boundary V,X)$ meant that any hierarchy path
connecting a bounded-length marking on $X$ to a marking containing
$\boundary V$ would have to pass through markings
containing $\boundary U$. 

Thus, informally (2) says that $U$ is between $V$ and $Y$, 
but (1) says that $U$ is {\em not}
between $V$ and $X$, and so on. 
Together these inequalities say that, in ``traveling'' from $Y$ to $X$, we
must first pass through $U$ and then through $V$.

Theorem \ref{length and projections} implies  that 
subsurfaces in $\LL_c(X,Y)$, for suitable $c$, have short boundary
curves in $Q(X,Y)$, and therefore are topologically ordered as
above. 
Lemmas 2.2 and 4.9 and the Bilipschitz Model Theorem from \cite{ELC2}
combine to show that, indeed,  the partial order $\prec$ determines the topological
ordering of the boundary components of the subsurfaces when $c$ is
large. In particular the combinatorial notion of ``betweenness'' translates
to a topological statement, that in a suitable product structure on
the manifold,  one level surface lies at a height between two others. The
following statement will suffice for us:

\begin{lemma}{porderandtoporder}
There exists $c_0>m_1$ such that if  $c>c_0$, $U,V\in \LL_c(X,Y)$,
and  $U\prec V$, then if  a boundary
component $\alpha$ of $U$ overlaps a boundary component 
$\beta$ of $V$, then $\alpha$ lies below $\beta$ in $Q(X,Y)$.
\end{lemma}

It is a simple observation that a curve $\alpha$ which is short in  the top conformal boundary
lies above any curve $\beta$ which is short in the manifold, if
$i(\alpha,\beta)> 0$. 

\begin{lemma}{boundary order}
If $l_{\alpha}(X)<\ep_0$ and $l_\beta(Q(X,Y))<\ep_0$ and $\alpha$ and
$\beta$ overlap
then $\alpha$ lies above $\beta$ in $Q(X,Y)$. Similarly, if $l_{\beta}(Y)<\ep_0$
and $l_\alpha(Q(X,Y))<\ep_0$ and $\alpha$ and $\beta$ intersect, 
then $\alpha$ lies above $\beta$ in $Q(X,Y)$

\end{lemma}

\begin{proof}{}
We give the proof in the case that $l_\alpha(X)<\ep_0$. 
A result of Epstein-Marden-Markovic \cite[Theorem 3.1]{EMM}
implies that $\alpha$ has length at most $2\ep_0$ in the top boundary component of the convex core of $Q(X,Y)$.
Therefore, one may isotope the geodesic representative of $\alpha$ onto
the top boundary component of the convex core entirely within the Margulis tube of $\alpha$.
One may then isotope it to $+\infty$ in the complement of the
convex core. The geodesic representative $\beta^*$ of $\beta$ is contained in
the convex core, and since it has length less than $\ep_0$ it is
contained in its own Margulis tube which is disjoint from that of $\alpha$.
It follows that the homotopy does not intersect $\beta^*$.
\end{proof}

The lemma below will be used in the proof of Theorem \ref{noselfbump}
in the surface group case to control the impact of changing the top conformal
structure on the ordering of the short curves and on related features. It is really just a
repackaging of the preceding sequence of lemmas. It says that, if $\alpha$ is
known to be short in $Q(X,Y)$, and $Z$ is ``between'' $\alpha$ and the top
conformal structure $X$ in the combinatorial sense discussed above,
then indeed $\boundary Z$ is also short in $Q(X,Y)$, and each of its
components that overlap $\alpha$ are topologically ordered above it. 

\begin{lemma}{alpha below Z}
There exists $d_0$ and $\delta_0\in (0,\epsilon_0)$  such that,
if $l_\alpha(Q(X,Y))<\delta_0$, $\alpha\in\mathcal{C}(S)$ overlaps $Z$
and $d_Z(X,\alpha) > d_0$, then  $l_{\partial Z}(Q(X,Y))<\delta_0$ and
each component of $\partial Z$ which overlaps $\alpha$ lies above $\alpha$. 
\end{lemma}

\begin{proof}{}
Applying Theorem \ref{length and projections} we may choose $\delta_0\in (0,\ep_0)$
so that
$$l_\alpha(Q(X,Y))<\delta_0 \implies
\mm_\alpha(X,Y) > \max\{c_0,2m_1\}.$$
Applying the other direction of Theorem \ref{length and projections},
we  choose $d_0>c_0+m_1+2$ so that  if $W\subset S$, then 
$$d_W(X,Y)>d_0-m_1-2 \implies
l_{\partial W}(Q(X,Y))<\delta_0.$$

We  note that 
if $l_\alpha(X)<\delta_0$, then $d_Z(X,\alpha)\le 2<d_0$,
so we may assume that $l_\alpha(X)\ge \delta_0$.

If $l_\alpha(Y)<\delta_0$, then Lemma \ref{boundary order} implies that
each component of $\partial Z$ which overlaps $\alpha$ lies above $\alpha$. 
Moreover, $d_Z(Y,\alpha)\le 2$, so 
$$d_Z(X,Y)\ge d_Z(X,\alpha)-d_Z(\alpha,Y)>d_0-2$$
so $l_{\partial Z}(Q(X,Y))<\delta_0$. This completes the proof in this case.

Hence we can now assume $l_\alpha(Y)\ge\delta_0$. Now $\mm_\alpha(X,Y)
> 2m_1$ implies
that there exists an essential subsurface $W\subset S$
with $\alpha\subset \partial W$ such that $d_W(X,Y)> 2m_1$. 
Since $d_Z(X,\partial W)>d_0-1>m_1$ and $d_W(X,Y)>2m_1$, 
Lemma \ref{partialorder}(6) implies that $W\prec Z$ in $\LL_{c-m_1}(X,Y).$
Lemma \ref{partialorder}(3) implies that $d_Z(Y,\partial W)\le m_1$.
Therefore, 
$$d_Z(X,Y)\ge d_Z(X,\partial W)-d_Z(\partial W,Y)>d_0-1-m_1,$$
so $l_{\partial Z}(Q(X,Y))<\delta_0$.
Lemma \ref{porderandtoporder} then implies that each component of $\partial Z$
which overlaps $\alpha$ lies above $\alpha$. 
\end{proof}

\subsubsection*{Predicting geometrically infinite ends in an algebraic limit}

Geometrically infinite surfaces in the algebraic limit can be detected
by looking at the limiting behavior of the ending invariants. Recall that
Masur and Minsky \cite{masur-minsky} proved that if $W$ is an essential subsurface of $S$, then $\mathcal{C}(W)$
is Gromov hyperbolic  and Klarreich \cite{klarreich} (see also
Hamenstadt \cite{hamenstadt}) proved that if
$W$ is not an annulus or pair of pants, then its Gromov boundary
$\partial_\infty\mathcal{C}(W)$ is identified with $\EL(W)$.

\begin{theorem}{endlams}{\rm (\cite{pull-out})}{}
Let $\{\rho_n\}$ be a sequence in $AH(S)$ converging to $\rho$
such that the top ending invariant of $\rho_n$ is $X_n\in\mathcal{T}(S)$. If
$W$ is an essential subsurface of $S$, the following statements are
equivalent:
\begin{enumerate}
\item
$N_\rho^0$ has an upward-pointing end bounded by $W$ with ending lamination
$\lambda\in \EL(W)$. 
\item
$\{\pi_W(X_n)\}$ converges to $\lambda$.
\end{enumerate}
Moreover, if $\{\rho_n=Q(X_n,Y_n)\}$, then $\pi_W(Y_n)$ does not accumulate
at $\lambda$  if $\{\pi_W(X_n)\}$ converges to $\lambda$.

Similarly, we obtain an equivalence if upward is replaced by downward
and the roles of $X_n$ and $Y_n$ are interchanged. 
\end{theorem}

A key tool in the proof of Theorem \ref{endlams} is the fact  that for
any non-annular subsurface $W$  the set of bounded length
curves in $Q(X,Y)$ project to a set of curves in $\mathcal{C}(W)$ which
are a bounded Hausdorff distance from  any geodesic in $\mathcal{C}(W)$ 
joining $\pi_W(X)$ to $\pi_W(Y)$. This result will be itself used in the proof of
Lemma \ref{U is nbhd QF}. We state the results in the special case
of quasifuchsian groups.

\begin{theorem}{bounded curves near  geodesic}{\rm (\cite{pull-out})}{}
Given $S$, there exists $L_0>0$ such that for all $L\ge L_0$, there
exists $D_0$, such that, if $X,Y\in\TT(S)$,
$\rho = Q(X,Y)$, $W\subset S$ is an essential subsurface and
$$
C(\rho,L) = \{\alpha\in\CC(S):l_\alpha(\rho) < L\},
$$ 
then $\pi_W(C(\rho,L)\cap \CC(S,W))$
has Hausdorff distance  at most $D_0$ from any geodesic in $\CC(W)$ joining
$\pi_W(X)$ to $\pi_W(Y)$.
Moreover if $d_W(X,Y) > D_0$ then 
$$C(W,\rho,L)= \{\alpha\in\CC(W):l_\alpha(\rho) < L\}$$
is nonempty and also has Hausdorff distance  at most $D_0$ from any geodesic 
in $\CC(W)$ joining $\pi_W(X)$ to $\pi_W(Y)$.
\end{theorem}

\subsection{Geometric limits}
\label{geometric limits}

A sequence $\{\Gamma_n\}$ of torsion-free 
Kleinian groups {\em converges geometrically} to
a torsion-free Kleinian group $\Gamma$ if $\Gamma$ is the set 
of all accumulation points
of sequences of elements $\{\gamma_n\in\Gamma_n\}$ and every $\gamma\in\Gamma$ is a
limit of a sequence of elements $\{\gamma_n\in\Gamma_n\}$; or in other
words if $\{\Gamma_n\}$ converges
to $\Gamma$ in the sense of Gromov-Hausdorff as subsets of ${\rm Isom}_+({\bf H}^3)$.
One may equivalently express this in terms of Gromov convergence of
the quotient hyperbolic 3-manifolds (see \cite{BP} for example).
If $N_n={\bf H}^3/\Gamma_n$ and $N={\bf H}^3/\Gamma$ and $x_n$ and $x_0$
denote the projections of the origin in ${\bf H}^3$, then $\{\Gamma_n\}$
converges geometrically to $\Gamma$ if and only if there exists a nested sequence
of compact submanifolds $\{X_n\}$ of $N$ which exhaust $N$  and $K_n$-bilipschitz diffeomorphisms
$f_n:X_n\to Y_n$ onto submanifolds of $N_n$ such that $f_n(x)=x_n$, $\lim K_n=1$ and 
$f_n$ converges uniformly on compact subsets of $N$ to an isometry
(in the $C^\infty$-topology).

Lemma 3.6 and Proposition 3.8 of Jorgensen-Marden \cite{JM} guarantee that if
$\{\rho_n\}$ is a sequence in $AH(M)$ converging to $\rho$, then
there is a subsequence of $\{\rho_n(\pi_1(M))\}$ which converges
geometrically to a torsion-free Kleinian group $\hat\Gamma$ such
that $\rho(\pi_1(M))\subset\hat\Gamma$. 

We say that a sequence $\{\rho_n\}$ in $AH(M)$ converges {\em strongly}
to $\rho\in AH(M)$ if it converges in $AH(M)$ and $\{\rho_n(\pi_1(M))\}$
converges geometrically to $\rho(\pi_1(M))$.
One may combine work of Anderson and Canary with the recent
resolution of Marden's Tameness Conjecture to show that  in the absence of
unnecessary parabolics, algebraic convergence implies strong convergence
(see also Theorem 1.2 of Brock-Souto \cite{brock-souto}).

\begin{theorem}{MPstrong} Let $M$ be a compact 3-manifold and let
$\{\rho_n\}$ be a sequence in $AH(M)$ converging to $\rho$ in $AH(M)$.
If every parabolic element of $\rho(\pi_1(M))$ lies in a rank two free abelian
subgroup, then $\{\rho_n\}$ converges strongly to $\rho$.
\end{theorem}

\begin{proof}{}
Theorem 3.1 of Anderson-Canary \cite{AC-cores2} and 
Theorem 9.2 of \cite{canary-cover} together
imply that if $\rho$ is topologically tame, then $\{\rho_n\}$ converges strongly to $\rho$.
The Tameness Theorem of Agol \cite{agol} and Calegari-Gabai \cite{calegari-gabai}
assures that $\rho$ is topologically tame, so our convergence is indeed strong.
\end{proof}

Proposition 3.2 of Anderson-Canary-Culler-Shalen \cite{ACCS} shows that whenever the
algebraic limit is a maximal cusp
(i.e. geometrically finite  and quasiconformally rigid), then
the convex core of the algebraic limit embeds in the geometric limit.
Remark 3.3 points out that the same argument applies
whenever the algebraic limit is topologically tame and its convex core has totally
geodesic boundary. In particular, the result holds when the limit is quasiconformally rigid.

\begin{proposition}{convexembed}{}{}
If $\rho$ is a quasiconformally rigid point in $\partial AH(M)$
and $\{\rho_n\}$ converges algebraically to $\rho$
and $\{\rho_n(\pi_1(M))\}$ converges geometrically to $\hat\Gamma$,
then the convex core of $N_\rho$ embeds in $\hat N=\Hyp^3/\hat\Gamma$
under the obvious covering map.
\end{proposition}

Proposition \ref{convexembed} will be used in section 3 to rule out bumping at quasiconformally
rigid points.
We will also use it to control the relative placement of closed curves in manifolds
algebraically near to a quasiconformally rigid manifold. Lemma \ref{hotsidehot} will only be needed in the quasifuchsian case discussed in section \ref{general surface}.

\begin{lemma}{hotsidehot}{}{}
If $\rho\in AH(S)$ is quasiconformally rigid,
$\alpha$ is an upward-pointing cusp in  $N_\rho$ and
$\beta$ is a  downward-pointing cusp in  $N_\rho$,
and $\alpha$ and $\beta$ intersect in $S$, then
there exists a neighborhood $U$ of $\rho$ in $AH(S)$ such that if $\rho'\in U$, then
$\alpha$ lies above $\beta \in N_{\rho'}$.
\end{lemma}

\begin{proof}{}
Find an embedded surface $F$  in $C(N_\rho)$ which is a compact core for
$C(N_\rho)$. Let $\epsilon<\epsilon_0$
be a lower bound for the injectivity radius of $N$ on $F$. Let $A$ be an embedded
annulus in $C(N_\rho)$,
intersecting $F$ only in one boundary component and whose other boundary
component is  curve in the homotopy class of $\alpha$ 
with length at most $\epsilon/4$.
Let $B$ be an embedded annulus in $C(N_\rho)$,
intersecting $F$ only in one boundary component and whose other boundary
component is  curve in the homotopy class of $\alpha$ with length at most $\epsilon/4$.

If the lemma fails we may produce a sequence
$\{\rho_n\}$  converging to $\rho$ such that $\alpha$ does not lie above $\beta$ in any $N_{\rho_n}$.
We may again pass to a subsequence such that $\{\rho_n(\pi_1(M))\}$ converges geometrically to $\hat\Gamma$ and 
$\rho(\pi_1(M))\subset \hat\Gamma$.  Let $\hat N={\bf H}^3/\hat\Gamma$
and let $\pi:N_\rho\to \hat N$ be the natural covering map.

By Proposition \ref{convexembed},
$\pi$ embeds $C=F\cup A\cup B$  in $\hat N$.
Then, for all large enough $n$, one can pull
$C$ back to $C_n=F_n\cup A_n\cup B_n$ by an orientation-preserving 2-bilipschitz map and $F_n$ is a compact core
for $N_{\rho_n}$ (as in the proof of Proposition 3.3 in Canary-Minsky \cite{canary-minsky}). One may join the
geodesic representative of $\alpha$  in $N_{\rho_n}$ to $\partial A_n-F_n$ 
by an annulus contained entirely
within the $\epsilon/2$-Margulis tube associated to $\alpha$. It follows that
this annulus cannot intersect $F_n$ (since $F_n$ is contained entirely in the $\epsilon/2$-thick part
of $N_{\rho_n}$) so we see that the geodesic representative
of $\alpha$ in $N_{\rho_n}$ lies above $F_n$. Similarly, the geodesic representative of $\beta$ in
$N_{\rho_n}$ lies below $F_n$. Therefore, for sufficiently large $n$, $\alpha$ lies above $\beta$
in $N_{\rho_n}$. This contradiction establishes the result.
\end{proof}

\section{Ruling out bumping}

In this section, we will show that there is no bumping at points
with no unnecessary parabolics or at quasiconformally rigid points.
The first case gives the non-bumping portion of Theorem \ref{nocuspscase},
while the second case is Theorem \ref{nobumping}.
In each case, we do so by showing that the (marked) homeomorphism
type is locally constant at $\rho$, which immediately implies that there is no
bumping at $\rho$ .
Note that in this section it will never be necessary to assume that $M$
has incompressible boundary.

The case where $\rho$ contains no unnecessary parabolics is especially
easy, since any sequence converging algebraically to $\rho$ converges
strongly.

\begin{proposition}{noparnobump}{}{}
Let $M$ be a compact 3-manifold and  $\rho\in \partial AH(M)$.
If every parabolic element of $\rho(\pi_1(M))$ lies in a rank two free abelian
subgroup, then $\Theta$ is locally constant at $\rho$. In particular, there is
no bumping at $\rho$.
\end{proposition}

\begin{proof}{}
Let  $\{\rho_n\}$ be a sequence  in $AH(M)$ which converges to $\rho$.
Theorem \ref{MPstrong} implies that $\{\rho_n\}$ converges strongly to $\rho$.
Results of Canary-Minsky \cite{canary-minsky} and
Ohshika \cite{ohshika-limits}, then imply that for all large enough $n$ there exists
a homeomorphism $h_n:N_\rho \to N_{\rho_n}$ in the homotopy class determined
by $\rho_n\circ \rho^{-1}$. It follows that $\Theta(\rho_n)=\Theta(\rho)$ for all large enough $n$, which completes the proof.
\end{proof}

\noindent
{\bf Remark:}  If we assume that $\{\rho_n\}\subset {\rm int}(AH(M))$, then strong
convergence follows from Theorem 1.2 of Brock-Souto \cite{brock-souto}. Consideration
of this case would suffice to establish that there is no bumping at $\rho$.

\medskip

If $\rho$ is a quasiconformally rigid point in $\partial AH(M)$, then sequences
of representations converging to $\rho$ need not converge strongly. However,
by Proposition \ref{convexembed}, the convex core of $N_\rho$ embeds in the geometric
limit of any sequence in $AH(M)$ converging to $\rho$, which will
suffice to complete the proof.
Proposition \ref{lcatqr} immediately implies Theorem \ref{nobumping}

\begin{proposition}{lcatqr}{}{}
If $M$ is a compact 3-manifold and  $\rho\in \partial AH(M)$ is quasiconformally rigid,
then $\Theta$ is locally constant at $\rho$. In particular, there is no bumping at $\rho$.
\end{proposition}

\begin{proof}{}
If $\Theta$ is not locally constant, then there is a sequence $\{\rho_n\}$ such that
$\Theta(\rho_n)\ne \Theta(\rho)$ for all $n$. 
We may pass to a subsequence, still called $\{\rho_n\}$, such that
$\{\rho_n(\pi_1(M))\}$ converges geometrically to $\hat\Gamma$ and 
$\rho(\pi_1(M))\subset \hat\Gamma$.  Let $\hat N={\bf H}^3/\hat\Gamma$
and let $\pi:N_\rho\to \hat N$ be the natural covering map. Proposition
\ref{convexembed} implies that $\pi$ embeds the convex core $C(N_\rho)$ into $\hat N$.

Let $C$ be a compact core for $C(N_\rho)$. 
We recall that for all sufficiently large $n$, there exists a $K_n$-blipschitz
diffeomorphism $f_n:X_n\to N_{\rho_n}$
from a compact submanifold $X_n$ of $\hat N$ which contains
$\pi(C)$ onto a compact submanifold of $N_{\rho_n}$.
The arguments of Proposition 3.3 of Canary-Minsky  \cite{canary-minsky} go through directly to
show that, again for large enough $n$, $C_n=f_n(\pi(C))$ is a compact core for $N_{\rho_n}$.
Moreover,  $(f_n\circ \pi)_*:\pi_1(C)\to \pi_1(C_n)$ is the same isomorphism, up to conjugacy,
induced by $\rho_n\circ\rho^{-1}$.  It follows that
$\Theta(\rho_n)=[(C_n,h_{\rho_n})]=[(C,h_\rho)]=\Theta(\rho)$. 
\end{proof}

\section{Ruling out self-bumping in the absence of parabolics}

In this section we rule out self-bumping at points in $\partial AH(M)$ with
no unnecessary parabolics when $M$ has incompressible boundary.
Proposition \ref{noparnobump} and \ref{noparnoselfbump} combine
to establish Theorem \ref{nocuspscase}.

\begin{proposition}{noparnoselfbump}{}{}
Let $M$ be a compact 3-manifold with incompressible boundary and
$\rho\in \partial AH(M)$.
If every parabolic element of $\rho(\pi_1(M))$ lies in a rank two free abelian
subgroup, then there is no self-bumping at $\rho$.
\end{proposition}

\begin{proof}{}
Let $M_\rho$ be a relative compact core for $N_\rho^0$ and let $\{S_1,\ldots,S_r\}$ denote
the non-toroidal components of $\partial M_\rho$.
We may order the boundary components so that 
$\{S_1,\ldots,S_k\}$ correspond to geometrically finite
ends of $N_\rho^0$ while $\{ S_{k+1},\ldots, S_r\}$ correspond to geometrically
infinite ends of $N_\rho^0$. Let $\{ \tau_1,\ldots,\tau_k,\lambda_{k+1},\ldots
\lambda_{r}\}$ be the end invariants of $\rho$ where $\tau_i\in \mathcal{T}(S_i)$
for all $i\le k$ and $\lambda_i\in\mathcal{EL}(S)$ for all $i> k$.

Let $B$ be the component of ${\rm int}(AH(M))$ corresponding to $[(M_\rho,h_\rho)]$.
Since $\Theta$ is locally constant at $\rho$, by Proposition \ref{noparnobump},
$B$ is the only component of ${\rm int}(AH(M))$ containing $\rho$ in
its closure.
We may identify $B$ with $\mathcal{T}(S_1)\times\cdots\times\mathcal{T}(S_r)$.
Let $\{\rho_n=(\tau_1^n,\ldots,\tau_r^n)\}$ be a sequence in $B$ converging to $\rho$.
Theorem \ref{endlams}
implies that $\{\pi_{S_i}(\tau_i^n)\}\subset \mathcal{C}(S_i)$ converges to
$\lambda_k\in \partial_\infty \mathcal{C}(S_i)$ for all $i>k$. Theorem \ref{MPstrong}
implies that $\{\rho_n\}$ converges strongly
to $\rho$. Then, a result of Ohshika \cite{ohshika-caratheodory}
(see also Kerckhoff-Thurston \cite[Corollary 2.2]{kerckhoff-thurston})
implies that $\{\tau_i^n\}$ converges to $\tau_i$ for all $i\le k$.

Let $\{\rho_n=(\tau_1^n,\ldots,\tau_r^n)\}$ and
$\{\rho_n'=((\tau_1^n)',\ldots,(\tau_r^n)')\}$ be two sequences in
$B$ converging to $\rho$. In order to rule out self-bumping at $\rho$,
it suffices to construct  paths $\gamma_n$ in $B$ joining $\rho_n$ to $\rho_n'$
such that if $\nu_n\in\gamma_n$, then $\{\nu_n\}$ converges to $\rho$.
We choose $\gamma_n$ to be the Teichm\"uller geodesic in 
$\mathcal{T}(S_1)\times\cdots\times\mathcal{T}(S_r)$ joining $\rho_n$ to $\rho_n'$.
If $\{\nu_n=(\mu_1^n,\ldots,\mu_r^n)\in\gamma_n\}$ is a sequence,
then, for all $i\le k$, since both $\{\tau_i^n\}$ and $\{(\tau_i^n)'\}$
converge to $\tau_i$, $\{\mu_i^n\}$ also converges to $\tau_i$.
In \cite{masur-minsky} (see Theorems 2.3 and 2.6),
it is shown that a Teichm\"uller
geodesic in $\mathcal{T}(S_i)$ projects into a $c_2$-neighborhood of a geodesic
in $\mathcal{C}(S_i)$ (for some uniform choice of $c_2$). Therefore, since
$\{\pi_{S_i}(\tau_i^n)\}$ and $\{\pi_{S_i}((\tau_i^n)')\}$ both converge to
$\lambda_i\in\partial_\infty\mathcal{C}(S_i)$ for all $i>k$, we see that
$\{\pi_{S_i}(\mu^n_i)\}$ converges to $\lambda_i$ for all $i>k$.

If $M=S\times I$ for a closed surface $S$, then Thurston's Double Limit Theorem \cite{thurston2} implies
that every subsequence of $\{\nu_n\}$ has a convergent subsequence.
If $M$ is not homeomorphic to $S\times I$,
then the main result of Ohshika \cite{ohshika-qc}
(which is itself derived by combining results of Thurston \cite{thurston2,thurston3})
implies that every subsequence of  $\{\nu_n\}$ has a convergent subsequence.

Let $\nu$ be a limit of a subsequence of $\{\nu_n\}$, still
denoted $\{\nu_n\}$, in $AH(M)$. In order to complete the proof, it suffices
to show that $\nu=\rho$.  We do so by invoking the Ending Lamination Theorem.
The main difficulty here is that we do not know that $\nu$ does not contain
any unnecessary parabolics, so we cannot immediately conclude that $\{\nu_n\}$
converges strongly to $\nu$.

Let $h:M_\rho\to M_\nu$ be a homotopy equivalence such that
$h\circ h_\rho$ is homotopic to $h_\nu$.   Consider the sequence 
$\{\nu_n'=(\tau_1,\ldots,\tau_{k},\mu^n_{k+1},\ldots,\mu_r^n)\}$.
There exists a sequence of $K_n$-quasiconformal map conjugating
$\nu_n$ to $\nu_n'$ with $K_n\to 1$.
It follows that $\{\nu_n'\}$ also converges to $\nu$. Theorem 5 in Bers \cite{bers-slice}
implies that, for all $i\le k$,
the  sequence of components $\{\Omega^n_{i}\}$  of $\Omega(\nu_n')$
associated to $\nu_n(\pi_1(S_i))$ (where we have chosen a fixed subgroup in
the conjugacy class of subgroups associated to $\pi_1(S_{i})$) converges
in the sense of Caratheodory to a  component $\Omega_{i}$ of
$\Omega(\nu(\pi_1(S_i)))$ such that
$\Omega_{i}/\nu(\pi_1(S_{i}))$ is homeomorphic to $S_{i}$ with
conformal structure $\tau_{i}$.  It then follows (again from Ohshika
\cite{ohshika-caratheodory}) that $\Omega_i$ is a component of the domain of discontinuity
of any geometric limit of $\{\nu_n'(\pi_1(M))\}$. Therefore, $\Omega_i$ is
a component of $\Omega(\nu)$ and the stabilizer of $\Omega_i$ in $\nu(\pi_1(M))$
contains $\nu(\pi_1(S_i))$ as a finite index subgroup. Therefore, we may homotope
$h$ so that, for all $i\le k$, $h|_{S_{i}}$  is an orientation-preserving covering map of
a component of $\partial M_\nu$ which is locally  conformal.

If $i>k$, Theorem \ref{endlams} implies that the cover $(N_\nu)_i$ of
$N_\nu$ assocated to $\pi_1(S_i)$ has a geometrically infinite end $\tilde E_i$ with
ending lamination $\lambda_i$. Moreover, if the orientation on $S_i$ is chosen
so that the geometrically infinite end in  $M_\rho$ is upward-pointing,
then $\tilde E_i$ is also upward-pointing
in $(N_\nu)_i$.  The Covering Theorem
(see \cite{thurston-notes} or \cite{canary-cover}) then implies that the covering
map $p_i:(N_\nu)_i\to N_\nu$ is finite-to-one on a neighborhood of $\tilde E_i$.
Therefore, we may homotope $h$ so that $h|_{S_i}$ is an orientation-preserving covering map with image a component of $\partial M_\nu$. If $T_j$ is a toroidal component of $\partial M_\rho$, then, since all incompressible tori are peripheral
in $M_\nu$, $h|_{T_j}$ can again be homotoped to a covering map onto a toroidal
component of $\partial M_\nu$. Therefore, we may assume that $h$ is a covering
map on each component of $\partial M_\rho$ and is orientation-preserving on
each non-toroidal component.

Waldhausen's Theorem \cite{waldhausen} now implies that $h$ is homotopic
to an orientation-preserving homeomorphism
$h':M_\rho\to M_\nu$, by a homotopy keeping $h|_{\partial M_\rho}$
constant.
It follows that $(M_\nu,h_\nu)$ is equivalent to $(M_\rho,h_\rho)$
and that the ending invariants are identified. The Ending Lamination Theorem
then implies that $\nu=\rho$. It follows that $\{\nu_n\}$ converges to $\rho$
as desired.

\end{proof}

\section{Fenchel-Nielsen coordinates and projection coefficients}
\label{section: FN projections}

In this section we discuss and compare {\em length-twist parameters} for
$\TT(S)$. For traditional Fenchel-Nielsen twist parameters based on a
maximal curve system $\ba$ (also known as a pants decomposition),  we
will see how the twist parameters compare 
with coarse twist parameters coming 
from projections to the annulus complexes associated to each curve in
$\ba$. More generally for a curve system $\ba$ that may not be maximal,
Theorem  \ref{FN projections}
allows us to vary arbitrarily the length and twist parameters of a
curve system $\ba$, while (coarsely) fixing all subsurface projections in the
complement of $\ba$.

To state the main theorem of this section we fix notation for the parameter spaces as
follows.  Given a curve system $\ba=\alpha_1\union\cdots\union\alpha_m$, 
define $T_\ba = \R^m$, 
$L_\ba = \R_+^m$, and
$V_\ba = T_\ba \times L_\ba$. For each component $\alpha_j$ of $\ba$ we have a
geodesic length function $l_{\alpha_j}:\TT(S)\to\R_+$, and we let
$$
l_\ba:\TT(S)\to L_\ba
$$
denote $(l_{\alpha_1},\ldots,l_{\alpha_m})$.

\begin{theorem}{FN projections}
Let $\ba$ be a curve system in $S$. For any $X\in\TT(S)$ there is a
continuous map
$$\Phi:V_\ba \to \TT(S)$$
such that $X\in \Phi(V_\ba)$, and such that
\begin{enumerate}
\item $l_\ba\circ \Phi(\bt, \bl) = \bl$
\item $|\tw_\ba(X, \Phi(\bt, \bl)) - \bt| < m_2$
and 
\item
for any essential subsurface $W\subset S$ disjoint from $\ba$
(except annuli parallel to components of $\ba$),
$$
\diam_{\CC(W)}(\Phi(V_\ba)) < m_2
$$
\end{enumerate}
where $m_2$ depends only on $S$. 
\end{theorem}
We will precisely define $\tw_\ba$ below but roughly speaking it is an
$m$-tuple of signed distances between the projections to the annular
complexes associated to the curve system $\alpha$.

Throughout this section an inequality of the form
$|\bt| < K$ for an $m$-tuple $\bt$ refers to the sup norm on $\bt$, so that
we are just bounding each component individually.

Theorem \ref{FN coordinates} will state the special case of Theorem
\ref{FN projections} when $\ba$ is a maximal curve system, 
namely that Fenchel-Nielsen coordinates can be chosen so that their twist
parameters agree roughly with the parameters given by $\tw_\ba$.

At the end of the section we will prove Lemma \ref{W connected}, which
is a connectivity result for a region in $\TT(S)$ given by bounding
the lengths of a curve system and restricting the structures in
the complementary subsurfaces to certain neighborhoods of points at
infinity. This lemma will be used in the last steps of the proofs of
Theorems \ref{bersslice}, \ref{acylcase},  and \ref{qfcase}.

\subsection{Coarse twist parameters}
An annulus complex is quasi-isometric to $\Z$. This allows us to define a signed version of distance. If $\alpha$ is the core curve of an annulus $W$ we denote $\CC(\alpha)=\CC(W)$,
$\pi_\alpha=\pi_W$, $d_\alpha = d_W$, and $\CC(S,\alpha)=\CC(S,W)$.

Given two elements $a$ and $b$ in $\CC(\alpha)$ we let $i_\alpha(a,
b)$ be the algebraic interersection of $a$ and $b$.  We then define 
$$\tw_\alpha: \CC(S, \alpha) \times \CC(S, \alpha) \to \Z$$
by $\tw_\alpha(\gamma, \beta) = i_\alpha(\pi_\alpha(\gamma),
\pi_\alpha(\beta))$. (If $a$ and
  $b$ have endpoints in common then algebraic intersection number is
not w-defined. We  correct this, in this special case, by taking the algebraic
  intersection of arcs in the homotopy class of $a$ and $b$ with
  minimal {\em geometric} intersection.)

There are two important properties of
$\tw_\alpha$ that we will use repeatedly: 
\begin{enumerate}
\item $d_\alpha(\gamma, \beta) = |\tw_\alpha(\gamma, \beta)| + 1$ if
  $\gamma \neq \beta$, and
\item $|\tw_\alpha(\gamma, \beta) + \tw_\alpha(\beta, \zeta) - \tw_\alpha(\gamma, \zeta)| \le 1$.
\end{enumerate}
(see \cite{ELC1} \S 4 for closely related properties).

Recall, that in Section \ref{ending invariants and geometry}, we defined $\pi_\alpha(X)$,
for $X \in \TT(S)$, by setting $\pi_\alpha(X) = \pi_\alpha(\beta)$ 
where $\beta$ is a shortest curve in $X$ that intersects $\alpha$. 
Abusing notation, we define
$$\tw_\alpha: \TT(S) \times \TT(S) \to \Z$$
by letting $\tw_\alpha(X,Y) = \tw_\alpha(\pi_\alpha(X), \pi_\alpha(Y))$. As we saw in Section \ref{ending invariants and geometry} if $\beta$ and $\beta'$ are both shortest length curves in $X$ that cross $\alpha$ then 
$$|\tw_\alpha(\beta, \beta')| + 1 = d_{\alpha}(\beta, \beta') \le 2.$$ Therefore $\tw_\alpha$ is w defined up to a uniform bound.

Furthermore since length functions are continuous on $\TT(S)$ the
function on $\TT(S)$ which gives back the length of the shortest curve
that crosses $\alpha$ is continuous. We also note that that the length
spectrum on a hyperbolic surface $X$, the values of lengths of curves
on a $X$, is discrete. These two facts allow us to find a neighborhood
$U$ of $X$ in $\TT(S)$ such that for every $Y\in U$ any shortest
length curve in $Y$ that crosses $\alpha$ is also a shortest length curve in
$X$ which crosses $\alpha$. It follows that $\tw_\alpha$ is {\em coarsely continuous}: there
is a constant $C$ such that every pair $(X,Y) \in \TT(S) \times
\TT(S)$ has a neighborhood $U$ such that $\diam(\tw_\alpha(U))<C$.

If $\ba = \alpha_1 \cup \dots \alpha_m$ is a curve system then $\tw_\ba$ takes values in $\Z^m$.

\subsection{Earthquakes and Fenchel-Nielsen coordinates}
\label{fenchel-nielsen}
For a curve $\alpha$, $s\in\R$ and $X\in\TT(S)$, a {\em right
  earthquake} of magnitude $s$ along $\alpha$  is obtained by cutting $X$ along the
geodesic representative of $\alpha$ and shearing to the right by
signed distance $s$ before regluing (so negative $s$ corresponds to
left shearing). See \cite{wpt:earthquakes,kerckhoff:nielsen}.
Let $e_{\alpha,t}(X)$ denote the result of a right
earthquake of magnitude $t l_\alpha(X)$, so that in particular
$$
e_{\alpha,1}(X) = \dehntw_\alpha(X)
$$
where $\dehntw_\alpha$ is a left Dehn-twist on $X$. The equivalence of
left twists with right shears corresponds to the fact that a mapping
class $f$ acts on $\TT(S)$ by precomposing the marking with $f^{-1}$.

For a curve system $\ba$ and $\bt\in T_\ba$, with
components $t_{\alpha_j} = t_j$, note that the shears
$e_{\alpha_j,t_j}$ commute and define
$$
e_{\ba,\bt} = e_{\alpha_1,t_1}\circ\cdots\circ e_{\alpha_m,t_m}.
$$
This earthquake map defines a free action of $T_\ba$ on $\TT(S)$ which fixes the fibers of the length map $l_\ba$. 

Now suppose that $\ba$ is a maximal curve system. Then the action on the fibers is also transitive and gives $\TT(S)$ the structure of principal $\R^m$-bundle over $L_\ba$. A choice of section of this bundle determines {\em Fenchel-Nielsen coordinates} for $\TT(S)$. More explicitly if 
$$\sigma : L_\ba \to \TT(S)$$
is a section then we can define a {\em Fenchel-Nielsen map}
$$F: V_\ba \to \TT(S)$$ by
$$F(\bt, \bl) = e_{\ba,\bt}(\sigma(\bl)).$$
This map will be a homeomorphism and give {\em Fenchel-Nielsen coordinates} for $\TT(S)$.
There are a number of concrete constructions for sections and Fenchel-Nielsen coordinates, but none are particularly canonical.

\subsection{Proof of Theorem \ref{FN projections}}
In this subsection we reduce the proof of Theorem \ref{FN projections} to three lemmas. We will prove these lemmas in the sections that follow.

It is not hard to measure how the twist parameter changes under powers
of Dehn twists. In particular,
$$|\tw_\alpha(X, \theta^n_\alpha(X)) - n|$$
is uniformly bounded. Rather than prove this directly we replace the Dehn twist with the earthquake map which allows us to replace the integer $n$ with a real number $t$. The first lemma generalizes the above bound for Dehn twists and is considerably more subtle to prove.
\begin{lemma}{coarse twist and earthquake}
There exists a constant $m_3$ such that
$$|\tw_\ba(X, e_{\ba, \bt} (X)) - \bt| \le m_3.$$
\end{lemma}

Next we see that projections to subsurfaces disjoint from $\ba$ remain coarsely constant when we earthquake along $\ba$.
\begin{lemma}{complementary subsurface}
There exists an $m_4$ such that for any essential subsurface $W \subset S$ disjoint from $\ba$ (except annuli parallel to components of $\ba$), and any $\bt \in T_\ba$
$$d_W(X , e_{\ba, \bt}(X)) < m_4$$
where $m_4$ only depends on $S$.
\end{lemma}

Finally we will construct a section of the bundle $l_\ba : \TT(S) \to L_\ba$ such the projection of all subsurfaces disjoint from $\ba$ is coarsely constant.

\begin{lemma}{section}
There exists an $m_5$ depending only on $S$ such that the following holds.
For any $X\in\mathcal{T}(S)$ there exists a section
$$\sigma: L_\ba \to \TT(S)$$
such that $X \in \sigma(L_\ba)$ and if $W\subset S$ is an essential subsurface
disjoint from $\ba$, then 
$$\diam_{C(W)}(\sigma(L_\ba)) < m_5.$$
\end{lemma}

Assuming these three lemmas it is easy to prove Theorem \ref{FN projections}.
\begin{proof}[Proof of Theorem \ref{FN projections}] We define the map $\Phi$ by
$$\Phi(\bt, \bl) = e_{\ba, \bt}(\sigma(\bl))$$
where $\sigma$ is the section given by Lemma \ref{section}. In
particular $l_\ba\circ \sigma(\bl) = \bl$. Since the earthquake maps
fix the lengths of $\ba$ we also have $l_\ba \circ \Phi(\bt, \bl) =
\bl$ and (1) holds. 

Let $m_2 = \max\{m_3 + m_5, m_4 +m_5\}$.
Note that
$$|\tw_\ba(X, \Phi(\bt, \bl)) - \tw_\ba(X, \sigma(\bl)) - \tw_\ba(\sigma(\bl), \Phi(\bt,\bl))| \le 1.$$
Lemma \ref{section} implies that $|\tw_\ba(X, \sigma(\bl))| + 1 < m_5$ and Lemma \ref{coarse twist and earthquake} implies that $|\tw_\ba(\sigma(\bl), \Phi(\bt, \bl)) - \bt| \le m_3$. Therefore
$$|\tw_\ba(X, \Phi(\bt, \bl)) - \bt| < m_3 + m_5 \le m_2$$
proving (2).

Let $W\subset S$ be an essential subsurface in $S$ disjoint from $\ba$
which is not an annulus parallel to a component of $\ba$.
By Lemma \ref{section}, 
$$d_W(X, \sigma(\bl)) < m_5,$$
 and Lemma
\ref{complementary subsurface} implies that 
$$d_W(\sigma(\bl),
\Phi(\bt, \bl)) < m_4.$$
Therefore 
$$d_W(X, \Phi(\bt, \bl)) < m_4 + m_5 \le m_2$$
proving (3).
\end{proof}

\subsection{Comparing twist coefficients}

To prove Lemma \ref{coarse twist and earthquake} we need an effective method of calculating $\tw_\ba$.
The map $\tw_\ba$ can be difficult to compute because, unlike other
subsurface projections, it is defined by lifting curves to a cover
rather than restricting them to a subsurface. We now describe a method
for approximating $\tw_\ba$ by restricting the curves to an annular
neighborhood of $\alpha$ (See  Minsky
\cite{ELC1} for a similar discussion.)

First, recall there is a uniform way to choose a regular neighborhood of a geodesic
in a hyperbolic surface. Namely there is a function $w:\R^+\to\R^+$
such that, for a simple closed geodesic $\gamma$ of length $l$ in any
hyperbolic surface, the neighborhood of radius $w(l)$, which we call
$\collar(\gamma)$, is an embedded annulus, and moreover
\begin{enumerate}
\item $\collar(\gamma)\intersect\collar(\beta) = \emptyset$ whenever
  $\gamma\intersect\beta = \emptyset$,
\item The length $l'$ of each component of $\boundary\collar(\gamma)$
  satisfies
$$
\max(a_0,l(\gamma)) < l' < l(\gamma) + a_1
$$
where $a_0, a_1$ are universal positive constants.
\end{enumerate}
See e.g. \cite{buser}. 
We can also define $\collar(\gamma)$ for a boundary
component of a surface, and extend the definition to give horocyclic
neighborhoods of cusps (here $l=0$ and
$w=\infty$) by requiring that the boundary length of the neighborhood
be fixed. 
If $\ba$ is a curve system then $\collar(\ba) =
\union_{\alpha_j\in \alpha} \collar(\alpha_j)$.

If $\alpha$ is a single curve and $a$ and $b$ are properly embedded
arcs in $\collar(\alpha)$ let $i^\bc_\alpha(a, b)$ be their algebraic
intersection. (When $a$ and $b$ have common endpoints we
  modify the definition just as we did for $i_\alpha(a,b)$.) If
$\gamma$ and $\beta$ are simple closed curves on $S$ that intersect
$\collar(\alpha)$ essentially and minimally in their homotopy class
define $\tw_\ba^\bc(\gamma, \beta) = i^\bc_\alpha(a,b)$ where $a$ and
$b$ are components of $\gamma \cap \collar(\alpha)$ and $\beta \cap
\collar(\alpha)$, respectively. As usual this definition depends on
the choice of component but only up to a bounded amount. Note that
while $\tw_\ba(\gamma, \beta)$ only depends on the homotopy classes of
$\gamma$ and $\beta$, $\tw^\bc_\alpha(\gamma, \beta)$ depends strongly
on the choice of curves. However, as we will see in the next lemma if
$\gamma$ and $\beta$ satisfy certain geometric conditions then
$\tw^\bc_\alpha(\gamma, \beta)$ is a good approximation for $\tw_\ba(\gamma, \beta)$.

\subsubsection*{Notation} 
To prevent a proliferation of constants throughout the remainder of this section we will use the following notation. The expression $x\sim y$ means that $|x-y| < c$ for some constant $c$ that depends only on $S$. We write $x \qsim K y$ if the constant depends on $S$ and some other constant $K$. For example, if $f \sim 0$ then the quantity $|f|$ is uniformly bounded.

\medskip

\begin{lemma}{collar restriction}
Let $\alpha$ be a curve in a curve system $\ba$ on $S$ and $X\in\TT(S)$.
Let $\gamma$ and $\beta$ be simple closed curves which intersect $\collar(\alpha)$
nontrivially, so that all components of their intersections with
$\collar(\ba)$ and with \hbox{$S\smallsetminus\collar(\ba)$} are essential.

Further assume that every component of $\gamma \cap ( S \setminus
\collar(\ba))$ that is adjacent to $\collar(\alpha)$ has length $<L$,
and similarly for $\beta$. 
Then
$$ \tw_\alpha(\gamma, \beta)  \qsim L \tw_\alpha^\bc(\gamma, \beta).$$
\end{lemma}

\begin{proof}
We consider another measure of twisting.
For two intersecting simple closed curves $\alpha$ and $\beta$ and a hyperbolic
structure $X$, we define a geometric shear of $\beta$ about $\alpha$ in $X$,
$s_{\alpha,X}(\beta)\in\R$, as follows. Let $A$ be a lift of the
geodesic representative of $\alpha$ to $\Hyp^2$, let $B$ be a lift of
$\beta$ which crosses $A$, and let $s_{\alpha,X}(\beta)$ denote
$1/l_\alpha(X)$ times the signed distance along $A$ between the orthogonal projections
to $A$ of the endpoints of $B$.  The sign is chosen so that
a left-earthquake of $X$ along $\alpha$ will increase
$s_{\alpha,X}(\beta)$. 

Since any two lifts of $\beta$ are disjoint, the values they give for
$s_{\alpha,X}$ differ by at most 1
(see Farb-Lubotzky-Minsky \cite{farb-lubotzky-minsky} for a discussion along these
lines). Moreover, $s_{\alpha,X}$ measures roughly the (signed) number of
fundamental domains of $\alpha$ crossed by the lift of $\beta$, and
this means that a difference of shears 
$s_{\alpha,X}(\gamma) -s_{\alpha,X}(\beta)$ 
coarsely measures the algebraic intersection numbers of lifts of
$\gamma$ and $\beta$  to the annulus cover associated to $\alpha$. 
In other words, comparing this with the
definition of $\tw_\ba$ we can see that, for any
$X,\alpha$, and $\gamma,\beta$ both crossing $\alpha$,
\begin{equation}\label{shear difference is twist}
\tw_{\alpha}(\gamma,\beta) \sim s_{\alpha,X}(\beta) -
s_{\alpha,X}(\gamma).
\end{equation}

We now make a similar definition using only $\collar(\alpha)$. Let
$\bC$ be a neighborhood of $A$ in $\Hyp^2$ that is a lift of
$\collar(\alpha)$ and consider the arc $B\intersect \bC$. Let
$s^\bc_{\alpha, X}(\beta)$ denote $1/l_\alpha(X)$ times the signed
distance along $A$ between the orthogonal projections to $A$ of the
endpoints of $B\intersect \bC$. As for $s_{\alpha, X}(\beta)$ the
signs are chosen so that a left-earthquake of $X$ along $\alpha$ will
increase $s^\bc_{\alpha, X}(\beta)$. Using the same reasoning as above
we see that
$$\tw^\bc_\alpha(\gamma,\beta) \sim s^\bc_{\alpha,X}(\beta) - s^\bc_{\alpha,X}(\gamma).$$

Note that $s_{\alpha,X}(\beta)$ only depends on the homotopy class of $\beta$ and the choice of lift. On the other hand, $s^\bc_{\alpha,X}(\beta)$ depends strongly on the curve $\beta$. However, given the restrictions we have put on $\beta$ we claim
\begin{equation}\label{estimate s}
s_{\alpha,X}(\beta) \qsim L s^\bc_{\alpha,X}(\beta).
\end{equation}
The lemma follows from this estimate.

To establish claim \ref{estimate s} we further examine the lift $B$ of $\beta$. 
Let $x^\bc$ be an endpoint of $B\intersect \bC$. After leaving $\bC$
at $x^\bc$, $B$ must continue to another lift $\bD$ of a component of
$\collar(\ba)$, and terminate at infinity at a point $x$ on the other
side of $\bD$. The distance in $\partial \bC$ between $x^\bc$ and the
orthogonal projection of $x$ to $\partial \bC$ will be bounded by $L$
plus the diameter of the projection of $\bD$.  The latter projects to
at most one fundamental domain of $\bC$ because the collars of $\ba$
are embedded. The arc of length $L$ projects, on the boundary of
$\bC$,  to at most $L/a_0$
fundamental domains because the length of each of them is at least
$a_0$. The bound of $1+L/a_0$ fundamental domains therefore applies to the projection to
the axis $A$ as well.
Applying the same estimate to the other endpoints, 
(\ref{estimate s}) follows.
\end{proof}

We can now prove Lemma \ref{coarse twist and earthquake}.

\begin{proof}[Proof of Lemma \ref{coarse twist and earthquake}]
We first assume that $\ba$ is a maximal curve system.
Let $\alpha_j$ be a curve in $\ba$ and let $\beta$ be a shortest curve
in $X$ that crosses $\alpha_j$, chosen so that
$\pi_{\alpha_j}(\beta) = \pi_{\alpha_j}(X)$.
Note that
$\collar(\ba)$ has a canonical affine structure given by the orthogonal foliations consisting of vertical geodesics orthogonal to core geodesics and horizontal curves equidistant to the core curve. There is then a canonical map from $X$ to $e_{\ba,\bt}(X)$ that is an isometry on $X \setminus \collar(\ba)$ and is an affine shear on each component of $\collar(\ba)$. Let $\beta'$ be the image of $\beta$ under this map and let $\gamma = \pi_{\alpha_j}(X)$ be a shortest curve in $e_{\ba, \bt}(X)$ that crosses $\alpha_j$. Then $\tw_{\alpha_j}(X, e_{\ba, \bt}(X)) = \tw_{\alpha_j}(\beta, \gamma)$.

Since $\ba$ is maximal and $\beta$ is a shortest curve that crosses $\alpha_j$, the length of every component of $\beta \setminus \collar(\ba)$ in $X$ is uniformly bounded.
It follows that every arc in $\beta' \setminus \collar(\ba)$ is uniformly bounded in $e_{\ba, \bt}(X)$. Similarly every component of $\gamma \setminus \collar(\ba)$ has uniformly bounded length in $e_{\ba, \bt}(X)$. Therefore we can apply Lemma \ref{collar restriction} to $\beta'$ and $\gamma$.

Since $\beta$ is a shortest curve crossing $\alpha_j$ in $X$ there is a vertical arc $b$ in $\collar(\alpha_j)$ that is disjoint from a component of $\beta \cap \collar(\alpha_j)$. Let $b'$ be the image of $b$ under the affine shear determined by $e_{\ba, \bt}$. In particular $b'$ will be disjoint from a component of $\beta'$. Similarly there is a vertical arc $a$ disjoint from a component of $\gamma \cap \collar(\alpha_j)$. Therefore
$$|\tw^\bc_{\alpha_j}(\beta',\gamma) - i^\bc_{\alpha_j}(b',a)| \le 2.$$
From the construction of the earthquake map we also see that
$$|i^\bc_{\alpha_j}(b', a) - t_j| \le 1$$
and it follows that
$$|\tw^\bc_{\alpha_j}(\beta',\gamma) - t_j| \leq 3.$$

Lemma \ref{collar restriction} then gives us our desired estimate for $\tw_{\alpha_j}(X,e_{\ba,\bt}(X))$ and applying this estimate to each component of $\ba$ gives us the lemma when $\ba$ is maximal.

If $\ba$ is not maximal we extend it to a maximal system $\hat\ba$. Given \hbox{$\bt\in T_\ba$,}
we extend it to $\hat\bt\in T_{\hat\ba}$  by letting all the coordinates corresponding
to components of $\hat\ba-\ba$ be 0. We then have
\begin{eqnarray*}
|\tw_\alpha(X, e_{\ba, \bt}(X)) - \bt| &=& | \tw_\alpha(X, e_{\hat\ba, \hat\bt}(X)) - \bt| \\
& \le & |\tw_{\hat\alpha}(X,e_{\hat\ba, \hat\bt}(X)) - \hat\bt|.
\end{eqnarray*}
The desired bound then follows from the bound in the maximal case
since $e_{\hat\ba,\hat\bt}(X)=e_{\ba,\bt}(X)$.
\end{proof}

We can now prove a special case of Lemma \ref{section} when $\ba$ is a maximal curve system. This special case is required to prove the more general version of the lemma.

\begin{lemma}{pants section}
Let $\ba$ be a maximal curve system on $S$ and let $X \in \TT(S)$.
Then there exists a section
$$\sigma: L_\ba \to \TT(S)$$
such that $X \in \sigma(L_\ba)$ and
$$|\tw_\ba(X,Y)| \sim 0$$
for all $Y \in \sigma(L_\ba)$.
\end{lemma}

\begin{proof}
Let 
$$\hat\sigma: L_\ba \to \TT(S)$$
be an arbitrary choice of section. We will use Lemma \ref{coarse twist and earthquake} to
``twist'' $\hat\sigma$ to our desired section $\sigma$.

Define a function $g: L_\ba \to T_\ba$ by
$$g(\bl) = \tw_\ba(X, \hat\sigma(\bl)).$$
Since $\hat\sigma$ is
continuous and $\tw_\ba$ is coarsely continuous, the function $g$ is
coarsely continuous. Recall this means there exists a constant  $C>0$ such that any $\bl\in L_\ba$ has a neighborhood $U$ 
with $\diam(g(U)) < C$.  

In particular, there exists a continuous function $\hhat g:L_\ba\to
T_\ba$, such that $|g-\hhat g| < 2C$: Simply triangulate $L_\ba$
sufficiently finely, set $\hhat g = g$ on the 0-skeleton, and extend by affine
maps to each simplex. 

We now define $\sigma$ by setting
$$\sigma(\bl) = e_{\ba, -\hat g(\bl)}(\hat\sigma(\bl)).$$
Lemma \ref{coarse twist and earthquake} then implies that
$$|\tw_\ba(\hat\sigma(\bl), \sigma(\bl)) + \hat{g}(\bl)| < m_3.$$
Using the fact that
$$|\tw_\ba(X, \hat{\sigma}(\bl)) + \tw_\ba(\hat\sigma(\bl), \sigma(\bl)) - \tw_\ba(X,\sigma(\bl))| \leq 1$$
and the bound on the difference between $g$ and $\hat{g}$ we have
$$|\tw_\ba(X, \sigma(\bl))| < m_3 + 2C + 1.$$
\end{proof}

Note that if $\ba$ is a maximal curve system then Lemma
\ref{complementary subsurface} is vacuous. In particular we have already
proven Theorem \ref{FN projections} in this special case. As it
may be of independent interest we state it as a theorem here. 

\begin{theorem}{FN coordinates}
Let $\ba$ be a maximal curve system for $S$. For any $X \in \TT(S)$ there exist
Fenchel-Nielsen coordinates
$$F: V_\ba \to \TT(S)$$
such that
$$\tw_\ba(X, F(\bt, \bl)) \sim \bt$$
for all $\bt\in T_\ba$ and all $\bl\in L_\ba$.
\end{theorem}

\subsection{Proof of Lemma \ref{complementary subsurface}}

To prove Lemmas \ref{complementary subsurface} and \ref{section} we
need to control subsurface projections along subsurfaces in the
complement of the curve system $\ba$ as we twist along $\ba$ and as we
vary the length of $\ba$. The difficulty is that as we vary the lengths
of $\ba$ we can not hope to control the behavior of the collection of shortest curves,
especially when all components of $\ba$ are very long.
What we will do instead is control the lengths of arcs on
complementary subsurfaces and we will see that this is sufficient. The
following lemma contains a more precise statement. It will be used in
the proofs of both Lemmas \ref{complementary subsurface} and
\ref{section}. 

If $R\subset S$ is an essential non-anular subsurface and $X$ is a
given hyperbolic structure on $S$, let $R^\bc$ denote the component of
$S\setminus \collar(\boundary R)$ which is isotopic to $R$. 

\begin{lemma}{bounding projections}
Let $R \subseteq S$ be a non-annular essential subsurface and $W
\subseteq R$ an essential (possibly annular) subsurface nested in
$R$. Let $\kappa$ be an essential simple closed curve or properly
embedded arc in $R$  that intersects $W$ essentially and let $L>0$ be
a constant. If  $X$ and $Y$ are hyperbolic structures in $\TT(S)$ such
that the length of $\kappa\intersect R^\bc$ is bounded by $L$ in both $X$ and $Y$
then 
$$d_W(X,Y) \qsim L 0.$$
\end{lemma}

\begin{proof} 
We first extend $\kappa$ to an essential simple closed $\gamma$. If  both endpoints of $\kappa$ lie on components of $\del R$ that are on the boundary of the same component of $S \setminus R$ then we choose $\gamma$ such that $\gamma \cap R = \kappa$. If the endpoints are on the boundary of different components then we construct $\gamma$ such that $\gamma \cap R$ is the union of $\kappa$ and an arc parallel to $\kappa$. If $\kappa$ is a simple closed curve then $\gamma = \kappa$. In all cases each component of the restriction of $\gamma$ to $R^\bc$ has length bounded by $L$.

We first assume that $W$ is non-annular. Let $\beta$
 be a shortest curve in $X$ that intersects $W$ essentially, such that
 $\pi_W(\beta) = \pi_W(X)$. The
 restriction of both $\gamma$ and $\beta$ to $W^\bc$ will have
 uniformly bounded length and hence uniformly bounded
 intersection.
  Therefore  $\pi_W(\gamma)$
 and $\pi_W(\beta)$ have bounded intersection giving a uniform bound
 on $d_{\CC(W)}(\pi_W(\gamma), \pi_W(\beta))$. 

If $\beta'$  is a shortest curve in $Y$ that intersects $W$
essentially such that $\pi_W(\beta') = \pi_W(Y)$, the same argument
shows that $d_{\CC(W)}(\pi_W(\gamma), 
\pi_W(\beta'))$ is uniformly bounded. The triangle inequality then
implies that 
$$d_{\CC(W)}(\pi_W(\beta), \pi_W(\beta')) = d_W(X,Y)$$
is uniformly bounded which completes the proof in the non-annular case.

We now assume that $W$ is an annulus with core curve $\zeta$. 
Since each arc of $\gamma \cap \collar(\zeta)$ has length at most
$L$, the width of the collar is bounded from above, which gives a bound from
below on $l_\zeta(X)$. Together these bounds imply a bound
on the number of times a component of
$\gamma\cap\collar(\zeta)$ winds around $\collar(\zeta)$. 
(More concretely, it gives an upper bound on the absolute value of the algebraic intersection
number of the component with a geodesic arc in $\collar(\zeta)$ which is orthogonal to $\xi$.)
Let $\beta$ be a shortest curve in $X$ 
crossing $\zeta$, such that $\pi_\zeta(\beta) = \pi_\zeta(X)$. Since $l_\zeta(X)$ is bounded below, $l_\beta(X)$ is uniformly bounded above. Since $\beta$ is a shortest curve each arc in $\beta \cap \collar(\zeta)$ intersects each geodesic arc in $\collar(\zeta)$ which is
orthogonal  to $\zeta$ at most once. Therefore, there is a uniform bound on 
$|\tw_\zeta^\bc(\beta,\gamma)|$ (measured with respect to $X$).

If $\ba = \del R \union \zeta$ then every component of $\gamma \cap \collar(\ba)$ that is adjacent to $\collar(\zeta)$ has length bounded by $L$ so we can apply Lemma \ref{collar restriction} to conclude that 
$$\tw_\zeta(\beta,\gamma) \qsim L \tw_\zeta^\bc(\beta,\gamma) \qsim L
0.$$

Repeating the argument with a curve $\beta'$ that is shortest in $Y$,
such that $\pi_\zeta(Y) = \pi_\zeta(\beta')$, we get a bound on
$\tw_\zeta(\beta',\gamma)$, and the desired bound on
$\tw_\zeta(\beta,\beta') = \tw_\zeta(X,Y)$ follows. 
\end{proof}

Lemma \ref{complementary subsurface} now follows easily. The proof of
Lemma \ref{section} is more involved.

\begin{proof}[Proof of Lemma \ref{complementary subsurface}] Let $W$
be a non-annular subsurface in the complement of $\alpha$. Let
$\kappa$ be a shortest curve on $X$ that intersects $W$, so that
there is a uniform length bound on $\kappa$.
Since the earthquake map is an isometry on
$W$ we have the same length bound on the intersection of $\kappa$
with $W^\bc$ in the metric  $e_{\ba, \bt}(X)$. Therefore
by Lemma \ref{bounding projections}, $d_W(X, e_{\ba, \bt}(X))$ is
uniformly bounded. 

Now let $W$ be an annulus with core curve $\zeta$. 
Add $\zeta$ to $\ba$ to make a new curve
system $\hat\ba$ and let $\hat\bt \in L_{\hat\ba}$ be equal to $\bt$
on the original $\ba$-coordinates and $0$ on the
$\zeta$-coordinate. Then $e_{\hat\ba, \hat\bt}(X) = e_{\ba, \bt}(X)$.
The bound on $|\tw_\zeta(X,e_{\ba,\bt}(X))|$ now 
follows from Lemma \ref{coarse twist and earthquake}.
\end{proof}

\subsection{Geometry of pants}
Before we begin the proof of Lemma \ref{section}, we need
to make some geometric observations about pairs of
pants. These are fairly basic but we will take some care
because we need statements that will hold uniformly for curves of all lengths.

Let 
$Y$ be a hyperbolic pair of pants with geodesic boundary, and let
$l_1,l_2,l_3$ denote its boundary lengths (we allow 0 for a cusp). Recall that
$Y^\bc$ denotes $Y\setminus \collar(\boundary Y)$. Now for each
permutation $(i,j,k)$ of $(1,2,3)$,  call a properly embedded essential
arc in $Y^\bc$ of type $ii$ if both its endpoints lie on the 
$i^{th}$ boundary component, and of type $jk$ if its endpoints lie
in the $j^{th}$ and $k^{th}$ boundary components. Define:
\begin{itemize}
\item $x_i$ to be the length of the shortest arc of type $ii$,
\item $y_i$ to be the length of the shortest arc of type $jk$, and
\item $\Delta_i = \half(l_j + l_k - l_i).$
\end{itemize}

The following lemma encodes the fact that $y_i$ is estimated by
$\Delta_i$ when $\Delta_i > 0$, and $x_i$ is estimated by  $-\Delta_i$
when $\Delta_i< 0$ -- and that $\min(x_i,y_i)$ is always bounded
above. This is because $Y^\bc$ retracts uniformly
to a 1-complex whose combinatorial type and geometry are
(approximately) dictated by the numbers $\Delta_i$.

\realfig{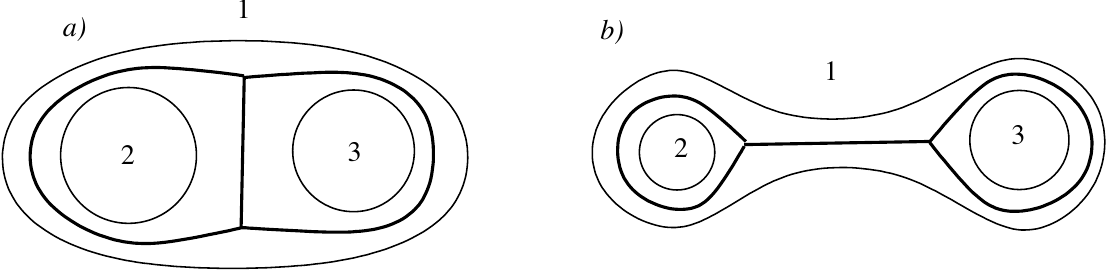}{The two types of thick hyperbolic pants
$Y^\bc$. In type {\em (a)}, the edges of the 1-complex have lengths
  $\Delta'_1, \Delta'_2$ and $\Delta'_3$.  In type {\em (b)}, the edge
  lengths are $-\Delta'_1, l'_2$ and $l'_3$.}

\begin{lemma}{pants geometry}
There exists $a>0$ such that, for a hyperbolic pair of pants labeled
as above,
\begin{align*}
\max(\Delta_i,0) - a &\le x_i \le 2\max(\Delta_i,0) + a \\
\intertext{and}
 \max(-\Delta_i,0) - a &\le y_i \le  \max(-\Delta_i,0) + a.\\
\end{align*}
\end{lemma}

\begin{proof}
(Sketch)
There is a subdivision (Voronoi diagram) of $Y^\bc$ into three convex annuli,
of width bounded by a uniform $w_1$, 
each containing the points closest to one of the boundary
components. The annuli meet in a geodesic 1-complex to
which $Y^\bc$ retracts. The $i$-th annulus is attached to the
1-complex along a curve whose length we denote by 
$l'_i$; note that $l_i<l'_i<l_i+a_2$ for a uniform
$a_2$. 

Now defining $\Delta'_i = \half(l'_j + l'_k - l'_i),$ it is
easy to see that the signs of the
$\Delta'_i$ (either all non-negative
or exactly one negative) determine the combinatorial type of this
1-complex: If all $\Delta'_i$
are nonnegative then the 1-complex is a ``theta'', three arcs attached
along endpoints so any two make a loop, and each $\Delta'_i$
is the length of the arc which, when deleted, leaves a loop homotopic
into the $i$-th annulus (see Figure \ref{pants} case (a)). 
If one $\Delta'_i<0$ then the 1-complex is a
``pair of glasses'', i.e. two disjoint loops homotopic into annuli $j$
and $k$ respectively and attached to the endpoints of
an arc, whose length is $-\Delta'_i$ (Figure \ref{pants} case (b)). 
Consider for example the theta case: each $x_i$ is bounded
above by $2w_1$, and each $y_i$ is between $-\Delta'_i$ and $-\Delta'_i +
2w_1$. The pair of glasses case is similar with a bit less symmetry,
accounting for the factor of 2 in the inequality.
Finally, the fact that $|\Delta_i-\Delta'_i| < 3a_2$ finishes the
proof.
\end{proof}

Let $\cP$ be a pair of pants and $\TT(\cP)$ the Teichm\"uller space of hyperbolic structures with geodesic boundary on $\cP$. We also allow the possibility that one or more of the boundary components is a cusp. The $l_i, x_i, y_i$ and $\Delta_i$ are now functions on $\TT(\cP)$. We also let $l_{ij}= (l_i,l_j)$ be the function which gives back the lengths of the $i$th and $j$th boundary component.

The following lemma should be thought of as a version of Lemma \ref{section} for pairs of pants.
\begin{lemma}{section over pants}
Given $s>0$ there exists an $s'$ such that the following holds. Let $Y$ be a hyperbolic structure in $\TT(\cP)$.
\begin{enumerate}
\item If $x_{1}(Y) < s$ then there exists a section $\sigma:[0,\infty) \to \TT(\cP)$ such that $l_1 \circ \sigma = \id$, $Y = \sigma(l_1(Y))$ and $x_1(Z) < s'$ for all $Z \in \sigma([0, \infty))$.

\item If $y_1(Y) < s$ there exists a section $\sigma:[0,\infty)^2 \to \TT(\cP)$ such that $l_{23} \circ \sigma =\id$, $Y = \sigma(l_{23}(Y))$ and $y_1(Z)<s'$ for all $Z \in \sigma([0,\infty)^2)$.
\end{enumerate}
\end{lemma}

\begin{proof}
We first prove (2). By Lemma \ref{pants geometry} we need to find a section such that the function $\max(-\Delta_1,0)$ is bounded on the image of $\sigma$. The Teichm\"uller space $\TT(\cP)$ is parameterized by the lengths of the boundary curves. This gives $\TT(\cP)$ a linear structure on which $\max(-\Delta_1, 0)$ is a convex. Triangulate $[0,\infty)^2$ with linear triangles and such that $l_{23}(Y)$ is a vertex in the triangulation. Define $\sigma(l_{23}(Y)) = Y$ and for any other vertex $v$ in the triangulation we define $\sigma(v)$ such that $\Delta_1(\sigma(v)) = 0$. We then extend $\sigma$ linearly across each triangle. By Lemma \ref{pants geometry}, $\max(-\Delta_i(Y), 0)$ is bounded by a constant only depending on $s$. On all other vertices $\max(-\Delta(\sigma(v)),0) = 0$. Therefore, by convexity,
$\max(-\Delta_1,0)\leq \max(-\Delta_1(Y),0)$ on the image of $\sigma$ as desired.

We can follow the same strategy to prove (1) except that now the triangulation of $[0,\infty)$ is just a partition into countably many compact segments.
\end{proof}

\begin{proof}[Proof of Lemma \ref{section}]
We will enlarge $\ba$ to a suitably chosen maximal curve system
$\hat{\ba}$ and write $L_{\hat{\ba}} = L_\ba \times L_{\hat{\ba}\setminus\ba}$. We will then define the section $\sigma$ by taking the section
$$\sigma_{\hat{\ba}}: L_{\hat{\ba}} \to \TT(S)$$
given by Lemma \ref{pants section}
and pre-composing it with a suitable  section
$$
\psi: L_\ba \to L_{\hat{\ba}}.
$$
That is we set $\sigma = \sigma_{\hat{\ba}} \circ  \psi$.

We will select $\hat{\ba}$ satisfying the following geometric properties: 
\begin{enumerate}
\item We can write $X\setminus\ba$ as
$$X\setminus \ba = Z_1\supset Z_2\supset\cdots \supset Z_k$$
such that, for $1\le i<k$,
$Z_{i+1}$ is obtained from $Z_i$ by cutting along a properly embedded
arc $\kappa_i$. More precisely, we will let $Y_i$ be a pair of pants
component of a regular neighborhood of $\kappa_i \union \boundary
Z_i$, and let $Z_{i+1} = int(Z_i \setminus Y_i)$. 
\item The boundary components of $Y_i$ that are incident to
  $\kappa_i$ are exactly those that are parallel to $\boundary Z_i$. 
\item The length of
$\kappa_i \intersect Y_i^\bc$ will be bounded by a uniform constant $b$.
\item  $Z_k$ will be a disjoint union of pairs of pants.
\end{enumerate}
See Figure
\ref{buildP} for an illustration of condition (2). 

The maximal curve system $\hat{\ba}$ will then be the union of $\ba$ with 
representatives of the isotopy classes of
the boundaries of the $Y_i$.

\realfig{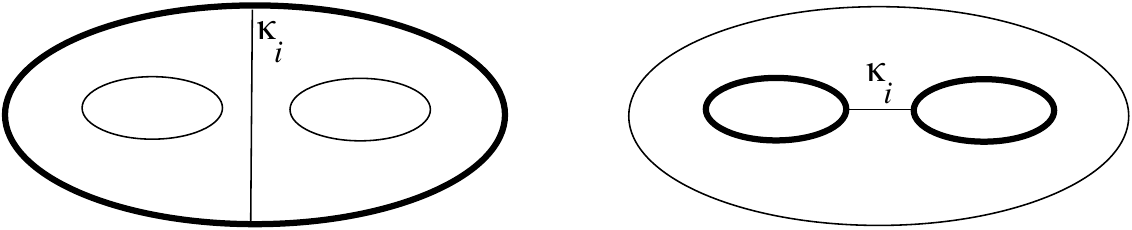}{The two configurations of $\kappa_i$ in $Y_i$ allowed
  by condition (2). Heavily shaded boundary components are parallel to
  $\boundary Z_i$.}

We proceed by induction. Let $Z_1 = X\setminus
\ba$, and let $U$ be a component of $Z_1$ which  is not a 3-holed
sphere. Because each component of $\boundary U^\bc$ has length
uniformly bounded below, area considerations give a uniform
$r_1$, such that the neighborhood of $\boundary U^\bc$ of radius $r_1$ cannot be
an embedded collar, and hence there is an essential properly embedded arc
$\kappa'_1\subset U^\bc$ of length
bounded by $b=2r_1$ (so we let $\kappa_1$ be properly embedded in
$U$ so that its intersection with $U^\bc$ is $\kappa'_1$). 
Now let $Y_1$ be the pair of pants obtained from a regular
neighborhood of $\kappa_1\union\boundary Z_1$, and let $Z_2 =
int(Z_1\setminus Y_1)$. Hence $\kappa_1\subset Y_1$ satisfies 
conditions (1) and (3) above. 

Since by construction both ends of $\kappa_1$ are on $\boundary U$ and
$\boundary Y_1$ must have at least one boundary component not parallel
to $\boundary U$, the only way
condition (2) can fail is if (numbering $\boundary Y_1$
appropriately and using the notation of Lemma \ref{pants geometry}),
$\kappa_1$ is of type 11 in $Y_1$,  while boundary component
number 2 is (isotopic to)  a boundary component of $U$. In this case,
we can replace the 11 arc by a 12 arc, whose length we can also bound.
Indeed,
note that the bound on $\kappa_1$ gives a bound on $\Delta_1$
by Lemma \ref{pants geometry}, and from the definition of the $\Delta_i$ we have 
that $-\Delta_3 \le \Delta_1$, so that again by 
Lemma \ref{pants geometry} we obtain a bound on the length
of the $12$ arc in $Y_1^\bc$. We
therefore replace $\kappa_1$ by the $12$ arc (keeping the name $\kappa_1$), noting that conditions (1), (2),
and (3) now hold. 

Now repeat inductively in $Z_i$ until we have reduced to a disjoint
union of pairs of pants.

\medskip

Having found $\hat{\ba}$ using this construction, we construct the section 
\hbox{$\psi:L_\ba \to L_{\hat{\ba}}$} -- that is, we build a continuous function $h:L_\ba\to
L_{\hat{\ba}\smallsetminus \ba}$. 
We will do this inductively, using Lemma \ref{section over pants}.
For $Y_1$, a point in $L_\ba$
determines the boundary lengths of the components of $\boundary Y_1$ that are adjacent to $\kappa_1$, and
the map given by Lemma \ref{section over pants} gives lengths for the remaining
components. For each successive $Y_i$, then, the already-defined
coordinates of $h$ determine the lengths for the components of
$\boundary Y_i$ that are adjacent to $\kappa_i$ (here we use property (2) of the
curve system $\hat{\ba}$), and the lemma again determines the rest. We then
 let $\psi$ be the section $\psi(\bl) =
(\bl,h(\bl))$.

It remains to verify that $\diam_{C(W)}(\sigma(L_\ba))$ is uniformly bounded for all subsurfaces $W \subset S$ that are disjoint from $\ba$.

If $W$ is an annulus whose core curve is a component of $\hat{\ba}$ then the bound follows from Lemma \ref{pants section} as the image of $\sigma$ lies in the image of $\sigma_{\hat{\ba}}$.

If $W$ is an annulus whose core curve is not in $\hat\ba$ or $W$ is non-annular, choose $Z_i$ such that $W \subseteq Z_i$ but $W \not\subseteq Z_{i+1}$. This implies that $\kappa_i$ intersects $W$ essentially and we can apply Lemma \ref{bounding projections} to $W$ and a component of $\kappa_i \cap W$ to obtain the bound.
\end{proof}

\subsection{Connectivity near infinity}
\label{connected infinity}

The following lemma will be used in the final steps of the proofs of
Theorems \ref{bersslice}, \ref{acylcase},  and \ref{qfcase}. 
It is a connectivity result for subsets of
Teichm\"uller space of the following type. Given a multicurve
$\ba$, let $\{S_1,\ldots,S_l\}$ be the
components of $S\smallsetminus\ba$ that are not 3-holed spheres,
select laminations $\lambda_i\in\EL(S_i)$, and let $U_i$ be
neighborhoods of $\lambda_i$ in $\CC(S_i)$ (recalling that $\EL(S_i)$
is the Gromov boundary of $\CC(S_i)$ by Klarreich's theorem). Let
$\bU$ denote the tuple $(U_i)$. Then define for $\ep<\ep_0$
\begin{align*}
\mathcal{T}(\ep,\bU)
=\{X\in\mathcal{T}(S)\ |\  & \pi_{S_i}(X)\in U_i\ \  \forall
i=1,\ldots,l, \\
& l_{\alpha_j}(X) < \ep \ \ \forall \alpha_j\in\ba\}.
\end{align*}

\begin{lemma}{W connected} Given a multicurve
  $\ba$ on $S$, let
  $\{S_1,\ldots,S_l\}$ be the components of $S\setminus\ba$ 
which are not thrice-punctured spheres. Given $\lambda_i\in \mathcal{EL}(S_i)$
and neighborhoods $U_i$ of $\lambda_i$ for all $i$ and $\ep < \ep_0$, there
exist neighborhoods $U'_i\subset U_i$ of $\lambda_i$  in $\mathcal{C}(S_i)$,
such that any two points in
$\mathcal{T}(\ep,\bU')$
are connected by a path in $\mathcal{T}(\ep,\bU)$. 
\end{lemma}

\begin{proof}
Let $\TT_\ep(\ba)$ denote the region of $\TT(S)$ where
$l_{\alpha_j} < \ep$ for all $\alpha_j\in\ba$. 

Recall that the Deligne-Mumford compactification of the Moduli space
of $S$ lifts to an ``augmentation'' of the Teichm\"uller space in
which a stratum $\TT_0(\gamma)$ is added for each curve system $\gamma$,
corresponding to ``noded'' Riemann surfaces where exactly the elements of
$\gamma$ are pinched, and parameterized by
$\TT(S\setminus\gamma)$. The topology of this bordification is the
smallest one for which
the length functions of simple closed curves, extended to allow the
value 0, are continuous (see e.g. Bers \cite{bers:noded}).

Extended Fenchel-Nielsen coordinates give us an explicit description of the
local topology at a stratum:  enlarge $\ba$ to a maximal curve system
$\hhat\ba$ and let $l_{\hhat\ba}$ and $t_{\hhat\ba}$ be
associated Fenchel-Nielsen length and twist parameters as in Section
\ref{fenchel-nielsen}. Adding $\TT_0(\ba)$ to $\TT(S)$ corresponds 
to enlarging the parameter spaces $V_{\hhat\ba} = T_{\hhat\ba}\times L_{\hhat\ba}$ to
allow points where $l_{\alpha_j}=0$ exactly for $\alpha_j\in\ba$, and then
taking a quotient by identifying points which agree on all coordinates except
possibly those $t_j$ for which $l_{\alpha_j}=0$ (in other words shearing around a
pinched curve is ignored). Let $\overline V_{\hhat\ba}$ denote this augmented
parameter space, which gives a homeomorphic model for
$\overline\TT^\ba \equiv \TT(S)\union
\TT_0(\ba)$. The map $V_{\hhat\ba} \to
V_{\hhat\ba\smallsetminus\ba}$ that forgets the $\ba$ coordinates
extends to a map of $\overline V_{\hhat\ba}$, and
gives us a retraction $\overline\TT^\ba \to \TT_0(\ba)$, which on
$\TT(S)$ is a fibration with contractible fibres. 

Because length functions are continuous in this topology, there is a
neighborhood $\VV_\ba$ of $\TT_0(\ba)$ in $\TT(S)$ for which the fibration, which we
write $\psi:\VV_\ba \to \TT_0(\ba) = \TT(S\smallsetminus\ba)$,
changes the lengths of, say, the set of shortest curves in the
complement of $\ba$ by a ratio arbitrarily close to 1.
Shrinking $\VV_\ba$ if necessary we may also assume
$\VV_\ba\subset\TT_\ep(\ba)$. The small perturbation of lengths
implies a distance bound in $\CC(S_i)$, namely
\begin{equation}\label{dSi bound}
d_{\CC(S_i)}(X,\psi(X)) < c_1,
\end{equation}
for a uniform $c_1$, when $X\in\VV_\ba$.

Using Theorem  \ref{FN projections}, for each $X\in \TT_\ep(\ba)$ one can find a
path $\{X_t\}$ in $\TT_\ep(\ba)$ connecting $X$ to a point
$X'\in\VV_\ba$, such that projections to each $\CC(S_i)$
remain uniformly bounded. (In fact this is just a pinching deformation
and the full power of Theorem  \ref{FN projections} is not needed.)
Hence we can and do choose $c_1$ sufficiently large that 
\begin{equation}\label{dSi bound 2} 
\diam_{S_i}(\{X_t\}) < c_1.
\end{equation}

In \cite{masur-minsky}, it is shown that a Teichm\"uller
geodesic in $\TT(S_i)$ projects to a $c_2$-neighborhood of a
$\CC(S_i)$-geodesic, with $c_2$ uniform.

Now by the definition of the Gromov boundary, there is a neighborhood
$U^0_i$ of each $\lambda_i\in\partial_\infty\CC(S_i)$ in $\CC(S_i)$
such that any $\CC(S_i)$-geodesic with
endpoints in $U^0_i$ has the property that its $c_2$-neighborhood is in
$U_i$. 

Let $U'_i$ be a neighborhood of $\lambda_i$  in $\CC(S_i)$ whose $2c_1$-neighborhood
is in $U^0_i$. Now suppose $X_1,X_2\in\Teich_\ep(\ba)$, and
$\pi_{S_i}(X_j) \in U'_i$ ($j=1,2$). Using (\ref{dSi bound 2}) we can
deform $X_j$ to $X'_j$ ($j=1,2$) within $\TT_\ep(S_i)$ so that $X'_j\in
\VV_\ba$ and the $\pi_{S_i}$-image of the path stays in a
$c_1$-neighborhood of $U'_i$. 
Then, by (\ref{dSi bound}),  $\pi_{S_i}(\psi(X'_j))\in
U^0_i$. Let $G$ be a Teichm\"uller geodesic in $\TT(S\setminus\ba)$
connecting
$\psi(X'_1)$ to $\psi(X'_2)$. Then $\pi_{S_i}(G)\subset U_i$, so a 
lift of $G$ back to $\VV_\ba$ with endpoints $X'_1$ and $X'_2$ will,
again by (\ref{dSi bound}),  give us the desired continuous family.  
\end{proof}

\section{Deformations with controlled projections}
\label{deformations}

In this section, we establish Lemma \ref{twist and shrink} which is a key technical
tool in the paper. We begin with a system of curves on the top conformal boundary
which are short in the manifold. Lemma \ref{twist and shrink} allows us to
to shrink the lengths of the curves on the top conformal boundary,
without disrupting the subsurface projections on complementary subsurfaces and keeping the
curves short in the manifold throughout the process.

\begin{lemma}{twist and shrink}
Given $S$ and $K>1/\ep_0$, there exists $c=c(S)$, depending only on $S$,
and $h=h(K,S)$, which depends on both $K$ and $S$, such that
if $X,Y\in \Teich(S)$ and $\ba$ is a curve system on $S$,  such that
$$
\mm_{\alpha_i}(X,Y) > h
$$
or 
$$
l_{\alpha_i}(X) < 1/K
$$
for each component $\alpha_i$ of $\ba$, then there exists a path
$\{X_t:t\in[0,T]\}$ in $\Teich(S)$ with $X_0=X$ such that
\begin{enumerate}
\item $l_{\alpha_i}(X_T) < 1/K$ for each $\alpha_i$,

\item  $\mm_{\alpha_i}(X_t,Y) > K$ for each $\alpha_i$ and each
  $t\in[0,T]$,
\item $\diam(\pi_W(\{X_t:t\in[0,T]\})) < c$, for any $W$ disjoint
  from $\ba$.
\end{enumerate}
\end{lemma}

Recall that, by  Theorem \ref{length and projections}, $\mm_\gamma(X,Y)$ is large if and only if $\gamma$ is short in $Q(X,Y)$. Lemma \ref{twist and shrink} will follow from Theorem \ref{FN projections}, which allows us
to change lengths and maintain control on subsurface projections, and Lemma
\ref{partialorder}, which records  key estimates concerning the partial order on subsurfaces of $S$.

\begin{proof} 
Let $\ep = 1/K$. 
Write $\ba= \ba^X \union \ba^Y\union\ba^0$,
where $\ba^X$ consists of those components $\alpha_i$ with
$l_{\alpha_i}(X) < \ep$, $\ba^Y$ consists of the components $\alpha_i$ of $\ba-\ba^X$
such that $l_{\alpha_i}(Y) < \ep$, and  $\ba^0$ consists of the remaining components. 

We argue by induction on
the cardinality $n$ of $\ba^0$ (which we note is bounded from above
in terms of $S$). We will iteratively construct $h_n$ (which implicitly depends on $K$ and $S$)
and show that the lemma holds if $\ba^0$ has $n$ components and
$\mm_{\alpha_i}(X,Y) > h_n$ for all $\alpha_i$ in $\ba^0$ with a constant
$c_n$ in (3), which only depends on $S$.

If $n=0$, we let $h_0=K$ and $c_0=m_2$ (the constant
from Theorem \ref{FN projections}).
If $\ba^Y=\emptyset$  then the deformation is trivial, i.e. $T=0$.

If $\ba^Y \ne\emptyset$, let $\Phi:V_\ba\to \TT(S)$ be the map given by
Theorem \ref{FN projections}, such that $X\in\Phi(V_\ba)$ -- in
fact we must have $X = \Phi(({\mathbf t},l_\ba(X)))$ 
for some $\mathbf t\in T_\ba$. Let $\{X_t:t\in[0,T_1]\}$ be the $\Phi$-image of
the path in $V_\ba$ that begins at $({\mathbf t},l_\ba(X))$, shrinks  the length of
each $\alpha_i$ in $\ba^Y$
monotonically to  $\ep/2$,  fixes the length of
every component of $\ba^X$ and fixes each twist coordinate. In particular,
$l_{X_T}(\alpha_i)<\epsilon$ for all $\alpha_i$ in $\ba$.
Since $l_{X_t}(\alpha_i)<\ep$  for all $\alpha_i$ in $\ba^X$ and
$l_Y(\alpha_i)<\ep$ if $\alpha_i$ lies in $\ba^Y$, we see immediately
$\mm_{\alpha_i}(X_t,Y)>K$ for all $i$ and $t$.
Theorem \ref{FN projections} also implies that
if $W$ is a subsurface disjoint from $\ba$, then
$$\diam(\pi_W(\{X_t:t\in[0,T]\})) < m_2=c_0.$$
The base case follows.

For $n>0$, 
set $h_n = h_{n-1} + 2m_1 + m_2 $ and $c_n=c_{n-1}+m_2$, where $m_1$ and $m_2$
are the constants from Lemma \ref{partialorder} and Theorem \ref{FN projections}. 

For each $\alpha_i$ in $\ba^0$, 
there must be some subsurface $W_i$
with $\alpha_i\subset \boundary W_i$, such that 
\begin{equation}\label{dW large}
d_{W_i}(X,Y) > h_n,
\end{equation}
since $\mm_{\alpha_i}(X,Y)>h_n>K$
but $l_{\alpha_i}(X) \ge \ep$ and $l_{\alpha_i}(Y) \ge \ep$.
(Note that possibly $W_i = \collar(\alpha_i)$). Fix one such $W_i$ for
each $\alpha_i\in\ba^0$. 

Since $h_n>m_1$, Lemma \ref{partialorder} implies that the set of
domains $\LL_{h_n}(X,Y)$, which contains all the 
$W_i$, is partially ordered by the relation $\prec$. 

Reordering $\ba$ if necessary,  we may assume that
$\alpha_1\in\ba^0$, 
and $W_1$ is $\prec$-maximal among the $W_i$, as well as maximal
with respect to inclusion among the $\prec$-maximal elements (Lemma
\ref{partialorder} implies that any two maximal elements are either
disjoint or nested).  In particular, the curves in $\partial W_1$ all lie
above any curves they intersect in $\partial W_i$, so, intuitively, $W_1$ is the closest
surface, among all the $W_i$, to the top of the manifold.
Let $\bb = \ba^X \union \{\alpha_1\}$. 

Now let $\Phi:V_\bb\to \TT(S)$ be the map given by
Theorem \ref{FN projections} such that $X=\Phi(\mathbf{t},l_\bb(X))$ for some
$\mathbf{t}\in T_\bb$.
Let $\{X_t:t\in[0,T_1]\}$ then be the $\Phi$-image of
the path in $V_\bb$ that begins at $({\mathbf t},l_\bb(X))$, shrinks the length of
$\alpha_1$ monotonically to $\ep$,  keeps the lengths of each element of $\ba^X$ fixed, and fixes each twist coordinate of an element of $\bb$.
Theorem \ref{FN projections} guarantees that 
\begin{equation}\label{W doesnt move}
\diam(\pi_W(\{X_t:t\in[0,T_1]\})) \le m_2
\end{equation}
if $W$ is disjoint from $\bb$ (including the case when $W$ is an annulus with
core a component of $\bb$). If $W$
intersects one of the curves of $\ba^X$, then since their lengths
are bounded by $\ep$ over the family $X_t$, we again have
a bound on $\diam_W(\{X_t\})$, by a constant which we may assume is 
smaller than $m_2$.
It follows, for any $W_i$ disjoint from $\alpha_1$,
that for all $t\in[0,T_1]$
\begin{equation}\label{Wi stays big}
d_{W_i}(X_t,Y) \ge h_n - m_2 = h_{n-1} + 2m_1.
\end{equation}
In particular, $d_{W_1}(X_t,Y)\ge h_n-m_2$ for all $t$. More generally, we see that
$\mm_{\alpha_i}(X_t,Y)>h_{n-1}$ for all $t$, whenever $W_i$ is disjoint from $\alpha_1$.

If $W_i$ intersects $\alpha_1$ then, by the choice of $W_1$, we see that 
$W_i$ and $W_1$ overlap and $W_i \prec W_1$,
with respect to the order $\prec$ on $\LL_{h_n}(X_0,Y)$. 
Lemma \ref{partialorder}(2) implies that 
$$d_{W_i}(Y,\boundary W_1) > h_n-m_1 > m_1.$$
Then, since $d_{W_1}(X_t,Y)\ge h_n -m_2>2m_1$ for all $t$, $W_1$ overlaps $W_i$, and
$d_{W_i}(Y,\boundary W_1) > m_1,$
Lemma \ref{partialorder}(5) implies that
$$d_{W_i}(X_t,Y)\ge h_n-m_2- m_1.$$
In particular this implies that $\mm_{\alpha_i}(X_t,Y)>h_{n-1}$ for
all $t$ and all $\alpha_i$ in $\ba^0$.

We now have a family $\{X_t:t\in[0,T_1]\}$
such that the number of components $\alpha_i$ of $\ba$ with
$l_{\alpha_i}(X_{T_1})\ge \ep$ and $l_{\alpha_i}(Y)\ge \ep$
is at most $n-1$. Moreover, for each $\alpha_i$ either,
$\mm_{\alpha_i}(X_{T_1},Y) > h_{n-1}$,  $l_{\alpha_i}(X_{T_1}) < \ep$
or $l_{\alpha_i}(Y) < \ep$ and if $W$ is disjoint from $\ba$, then 
$$\diam(\pi_W(\{X_t:t\in[0,T_1]\})) \le m_2.$$

Now applying the inductive hypothesis to $X_{T_1}$, we 
can concatenate this family with one that shrinks the remaining
components of $\ba$ to have length at most $\epsilon$, so that
$\mm_{\alpha_i}(X_t,Y) > K$ for each $\alpha_i$ and each  $t$, and
$$\diam(\pi_W(\{X_t:t\in[0,T]\})) < c_n,$$
for any $W$ disjoint from $\ba$.

\end{proof}

\section{Bers slices}
\label{bers slice}

In this section, we prove Theorem \ref{bersslice} which we re-state here
for the reader's convenience.

\medskip\noindent
{\bf Theorem \ref{bersslice}.} {\em
Let $B$ be a Bers slice of $QF(S)$ for some closed
surface $S$.  If $\rho\in \partial B$ and
$\rho$ is quasiconformally rigid in $\partial B$, then
$B$ does not self-bump at $\rho$. In particular, its closure  $\bar B$ is locally 
connected at $\rho$.}

\medskip

We will begin by proving that there is no self-bumping at a maximal cusp in the
boundary of a Bers slice.
The proof in this case is much simpler but follows the same outline as
the proof of the general case.

\subsection{The maximal cusp case}

We first assume that $\rho$ is a maximal cusp in the boundary of a Bers slice 
$B=B_Y$ in $AH(S)$. Let $\ba$
be the maximal curve system on $S$ which is cusped in $N_\rho$. 

If $\{\rho_n\}$ is a sequence in $B_Y$, then $\{\rho_n\}$ converges to
$\rho$ if and only if $\lim l_{\rho_n}(\alpha_j)=0$ for all $\alpha_j\in\ba$.
(Theorem 5 of Bers \cite{bers-slice} implies that $B_Y$ has compact closure in $AH(S\times I)$ while Theorem 1 in
Maskit \cite{maskit-koebe} implies that a maximal cusp in $\partial B_Y$ is determined
by its parabolic elements.)
Therefore the sets
$$
U(\delta) = \{\rho'\in B_Y: l_{\alpha_j}(\rho') < \delta
\ \ \forall \alpha_j\in\ba\}.
$$
for $\delta>0$
give a neighborhood system for $\rho$ in $ B_Y$

We will show that for each $\delta>0$ there exists a neighborhood
$V$ of $\rho$ such that any two points in $ V\intersect B_Y$ are connected
by a path in $U_\delta \intersect B_Y$.
It then follows that there is no self-bumping at $\rho$.

First, let 
$$W(\ep)=\{Q(X,Y)\in B_Y :  l_{\alpha_j}(X)<\ep \ \ \forall \alpha_j\in\ba\}.$$
$W(\ep)$ is path-connected  for any $\ep>0$, because it is parametrized by a 
convex set in the Fenchel-Nielsen coordinates  for $\mathcal{T}(S)$. 
Bers' Lemma \ref{berslemma} implies that $W(\ep) \subset U(2\ep)$.
Hence it suffices to choose $V$ so that
any point in $V$ can be connected to $W({\delta/2})$ by a path in $U(\delta)$. 

Given $\delta>0$, Theorem \ref{length and projections} allows us to choose
$K$ such that, for any $X,Y \in \Teich(S)$, 
$$
\mm_\gamma(X,Y) > K  \ \implies \ l_\gamma(Q(X,Y)) < \delta.  
$$
We may moreover require that $K > 2/\delta$. 
Let $h=h(K,S)$ be the constant given by Lemma \ref{twist and shrink}.
Theorem \ref{length and projections} then gives 
$\delta'>0$ such that
$$
 l_\gamma(Q(X,Y)) < \delta' \ \implies \  \mm_\gamma(X,Y) > h.
$$
Now consider $V=U({\delta'})$. If $Q(X,Y)\in V$, then $ \mm_{\alpha_i}(X,Y) > h$
for all $\alpha_j\in \ba$, so Lemma \ref{twist and shrink} gives
a family $\{X_t\ |\ t\in[0,T]\}\subset\mathcal{T}(S)$ with $X_0=X$ such that, for
each $\alpha_j\in\ba$, 
\begin{enumerate}
\item
$\mm_{\alpha_j}(X_t,Y) > K$ for all $t\in[0,T]$, and
\item
$l_{\alpha_j}(X_T)\le \frac{1}{K}<\frac{\delta}{2}$.
\end{enumerate}
It follows, from (1), that
$Q(X_t,Y)\in U(\delta)$ for all $t$ and, from (2), that  $Q(X_T,Y)$ is contained in
$W(\delta/2)$. 

This completes the proof of Theorem \ref{bersslice} for maximal cusps. 

\subsection{General quasiconformally rigid points on the Bers boundary}

In order to prove that there is no self-bumping at quasiconformally rigid
points we must also  allow for geometrically infinite ends. Theorem 
\ref{endlams} and the Ending Lamination Theorem allow us to
use subsurface projections to construct a neighborhood system about
a general quasiconformally rigid point. Once we have constructed
this neighborhood system the control we obtained on subsurface projections
in Lemma \ref{twist and shrink} allows us to proceed much as in the proof
of the maximal cusp case.

If $\rho\in\partial B_Y$ is quasiconformally rigid, then its
geometrically infinite ends are associated with a disjoint
collection of  subsurfaces $\{ S_1,\ldots, S_l\}$ of $S$ and the cusps
are associated with a collection $\ba$ of disjoint simple closed curves
such that the components of $S\setminus\ba$ are precisely
the $S_i$ together with a (possibly empty) collection of
thrice-punctured spheres. Let $\{\lambda_1,\ldots,\lambda_l\}$ be the ending laminations
supported on $\{ S_1,\ldots,S_l\}$.

Let $U_i$ be a 
neighborhood of $\lambda_i\in\partial_\infty\CC(S_i)$ in $\CC(S_i)$
for each $i=1,\ldots,l$. We denote by $\bU$ the tuple
$(U_1,\ldots,U_l)$, and for $\delta>0$ we let
$\UU(\delta,\bU)$ be the set
\begin{align*}
\UU(\delta,\bU) = \{Q(X,Y) :\  & \pi_{S_i}(X)\in U_i \ \  \forall
i=1,\ldots,l, \\
& l_{\alpha_j}(Q(X,Y)) < \delta \ \ \forall \alpha_j\in\ba\}.
\end{align*}

Theorem \ref{endlams} and the Ending Lamination Theorem allow us to show
that the $\UU(\delta,\bU)$ give a neighborhood system for $\rho$ in $B_Y$.
However, we should note that the sets $\UU(\delta,\bU)$ will not in general
be open in $B_Y$.

\begin{lemma}{U is nbhd}
The sets $\UU(\delta,\bU)$, 
where $\delta$ varies in $ (0,\ep_0)$ and the $U_i$ vary over
neighborhoods of $\lambda_i$ in $\bbar\CC(S_i)$, 
are the intersections with $B_Y$ of a neighborhood system for $\rho$.
\end{lemma}

\begin{proof}
It suffices to show that a sequence $\{\rho_n = Q(X_n,Y)\}$ converges to
$\rho$ if and only if it is eventually contained in any
$\UU(\delta,\bU)$. 

Let $\{\rho_n\}$ be a sequence eventually contained in any
$\UU(\delta,\bU)$. Since $\bbar B_Y$ is compact, it suffices to
show that any accumulation point of $\{\rho_n\}$ is $\rho$. 
Therefore we may assume $\{\rho_n\}$ converges to $\rho'$.
By hypothesis, for each
$S_i$, $\{\pi_{S_i}(X_n)\}$ converges to $\lambda_i$.
Theorem \ref{endlams} now
implies that $S_i$ faces an upward-pointing end of $\rho'$ with ending
lamination $\lambda_i$, for each $i$. Since $\lim l_{\alpha_j}(\rho_n) =0$,
each $\alpha_j$ corresponds to a cusp of $\rho'$. Since
$\rho'\in\bbar B_Y$, it has a  downward pointing end associated to the full
surface $S$, with conformal structure $Y$ (see Bers \cite[Theorem 8]{bers-slice}).
Thus, each cusp of $\rho'$ is upward-pointing.
Therefore, the end invariants of
$\rho'$ are the same as those of $\rho$. By the Ending Lamination
Theorem, $\rho'=\rho$. 

In the other direction, suppose $\{\rho_n\}$ converges to $\rho$. Then
$\lim l_{\alpha_j}(\rho_n)=0$  for all $\alpha_j\in\ba$, by continuity of length, and
$\{\pi_{S_i}(X_n)\}$ converges to $ \lambda_i$ for all $i$, by Theorem \ref{endlams}. Hence
$\{\rho_n\}$ is eventually contained in any $\UU(\delta,\bU)$. 
\end{proof}

Let $\WW(\ep,\bU)$ denote a similarly-defined set  where the length bounds
on $\ba$ take place in the boundary structure  $X$, i.e.
\begin{align*}
\WW(\ep,\bU) = \{Q(X,Y) :\  & \pi_{S_i}(X)\in U_i \ \  \forall
i=1,\ldots,l, \\
& l_{\alpha_j}(X) < \ep \ \ \forall \alpha_j\in\ba\}.
\end{align*}
Notice that
$\WW(\ep,\bU)  = \{Q(X,Y): X\in\TT(\ep,\bU) \},$
where $\TT(\ep,\bU)$ is as in \S\ref{connected infinity}. 
By Bers' Lemma \ref{berslemma},
$\WW(\delta/2,\bU) \subset\UU(\delta,\bU)$. 

Theorem \ref{bersslice}  follows from the following lemma:

\begin{lemma}{U to W}
Given $\delta>0$ and neighborhoods $U_i$ of $\lambda_i$, there exists
$\ep>0$ and neighborhoods $V_i$ of $\lambda_i$ such that any two 
points in $\UU(\ep,\bV)$ can be connected 
by a path that remains in
$\UU(\delta,\bU)$.
\end{lemma}

\begin{proof}
By Theorem \ref{length and projections}, choose
$K$ such that
$$
\mm_\gamma(X,Y) > K \ \implies \ l_\gamma(Q(X,Y)) < \delta
$$
and also suppose  $K> 2/\delta$. 
Let $h=h(K,S)$ be the constant given by Lemma \ref{twist and shrink}, and
let $c=c(S)$ be the constant in part (3) of Lemma \ref{twist and shrink}. 
Lemma \ref{W connected} allows us to choose neighborhoods $W_i$
of $\lambda_i$ such that any two points in $\WW(\delta/2,\bW)$ are connected by a
path in $\WW(\delta/2,\bU)$.

Choose $\ep>0$ small enough that (again by Theorem \ref{length and projections})
$$l_{\gamma}(Q(X,Y))<\ep \ \implies \ \mm_{\gamma}(X,Y) > h.$$
Finally, choose neighborhoods $V_i$ of
$\lambda_i$ such that a 
$c$-neighborhood of $V_i$  in $\CC(S_i)$ is contained in $W_i$. 

Let $Q(X,Y)$ be in $\UU(\ep,\bV)$. 
Then, by our choice of $\ep$, $\mm_{\alpha_i}(X,Y) > h$ for each
component of $\ba$, and 
so Lemma \ref{twist and shrink} can be applied to give
a path $\{X_t:t\in[0,T]\}$ such that 
\begin{enumerate}
\item
$l_{\alpha_j}(X_T) < 1/K <\delta/2$ for all $\alpha_j\in\ba$, 
\item
$\mm_{\alpha_j}(X_t,Y) > K$ for all $\alpha_j\in\ba$ and all $t\in [0,T]$, and
\item
$\diam_{\CC(S_i)}(\pi_{S_i}(\{X_t\})) < c$ for all $i$.
\end{enumerate}
It follows immediately from (1) and (3), that
$Q(X_T,Y) \in \WW(\delta/2,\bW)$. 
Moreover, (2) implies that 
$l_{\alpha_j}(Q(X_t,Y)) < \delta$ for all $t$ and all $\alpha_j\in\ba$,  so, again applying (3), we
see that the entire path
$\{ Q(X_t,Y)\}$ lies in $\UU(\delta,\bW)$. 

This shows that any point in $\UU(\ep,\bV)$ can be connected to
$\WW(\delta/2,\bW)$ by a path in $\UU(\delta,\bW)$. 
Now since any two points in
$\WW(\delta/2, \bW)$ can be connected by a path in
$\WW(\delta/2,\bU)$, and since
$\WW(\delta/2,\bU) \subset \UU(\delta,\bU)$,  we
conclude that any two points in $ \UU(\ep,\bV)$ can be connected
by a path in $\UU(\delta,\bU)$. 
\end{proof}

\section{Acylindrical manifolds}
\label{acylindrical}

In this section, we rule out self-bumping at quasiconformally rigid
points in boundaries of deformation spaces of acylindrical 3-manifolds.
Thurston's Bounded Image Theorem allows us to use essentially the
same argument as in the Bers Slice case. Theorem \ref{acylcase} is the
special case of Theorem \ref{noselfbump} where $M$ is acylindrical.

\begin{theorem}{acylcase}{}{}
Let $M$ be an acylindrical compact 3-manifold. If $\rho$ is a quasiconformally
rigid point in $\partial AH(M)$, then there is no self-bumping at $\rho$.
\end{theorem}

\begin{proof}{}
If $B$ is a component of ${\rm int}(AH(M))$ then we may
identify $B$ with $\mathcal{T}(S)$ where $S=\partial_TM$
is the non-toroidal portion of $\partial M$. Explicitly, we identify $\nu\in B$ with
$\partial_cN_\nu$, regarded as a point in $\mathcal{T}(S)$.
Thurston's Bounded Image Theorem asserts
that the skinning map $\sigma:\mathcal{T}(S)\to\mathcal{T}(S)$ has bounded image.
Let $L$ be the diameter of $\sigma(\mathcal{T}(S))$.

We again begin by constructing a neighborhood system about $\rho$.
Suppose that $B$ is the component of ${\rm int}(AH(M))$
such that $\rho\in \partial B$. Let $(M_\rho,P_\rho)$ be a relative compact core for
$N_\rho$. Let $\{ S_1,\ldots S_l\}$
be the components of $M_\rho-P_\rho$  which are not thrice-punctured spheres.
Each $S_i$ may be thought of as a subsurface of $S$ and comes equipped
with an ending lamination $\lambda_i$. The annular components of $P_\rho$
are associated with a disjoint collection $\ba$ of
simple closed curves on $S$.  Since $M$ is acylindrical, 
$\Theta$ is locally constant (see \cite{ACM}),
so our identification of $B$ with $\mathcal{T}(S)$ is consistent with our identification
of $\partial M_\rho$ with $S$.

Let $U_i$ be a 
neighborhood of $\lambda_i\in\partial_\infty\CC(S_i)$ in $\CC(S_i)$
for each $i=1,\ldots,i$. We denote by $\bU$ the tuple
$(U_1,\ldots,U_l)$, and for $\delta>0$ we let
$\UU(\delta,\bU)$ be the set
\begin{align*}
\UU(\delta,\bU) = \{\nu\in B :\  & \pi_{S_i}(\partial_cN_\nu)\in U_i \ \  \forall
i=1,\ldots,l, \\
& l_{\alpha_j}(\nu) < \delta \ \ \forall \alpha_j\in\ba\}.
\end{align*}

Since $AH(M)$ is compact (see \cite{thurston1}) if $M$ is acylindrical,
the proof of Lemma \ref{U is nbhd} generalizes directly to give:

\begin{lemma}{acylnbhd}
The sets $\UU(\delta,\bU)$, 
where $\delta$ varies in $ (0,\ep_0)$ and the $U_i$ vary over
neighborhoods of $\lambda_i$ in $\CC(S_i)$, 
are the intersections with $B$ of a neighborhood system for $\rho$.
\end{lemma}

We again define a related set $\WW(\ep,\bU)$ where the length bounds
on $\ba$ take place in the conformal boundary:
\begin{align*}
\WW(\ep,\bU) = \{\nu\in B :\  & \pi_{S_i}(\partial_cN_\nu)\in U_i \ \  \forall
i=1,\ldots,l, \\
& l_{\alpha_j}(\partial_cN_\nu) < \ep \ \ \forall \alpha_j\in\ba\}.
\end{align*}
Again Bers' Lemma \ref{berslemma} implies that
$\WW(\delta/2,\bU) \subset\UU(\delta,\bU)$. 

The proof of Theorem \ref{acylcase} is completed by
Lemma \ref{U to W acyl} whose proof mimics that of Lemma \ref{U to W}
but must be adapted to account for the fact that $\sigma$ is bounded
rather than constant.

\begin{lemma}{U to W acyl}
Given $\delta>0$ and neighborhoods $U_i$ of $\lambda_i$, there exists
$\ep>0$ and neighborhoods $V_i$ of $\lambda_i$ such that any two 
points in $\UU(\ep,\bV)$ can be connected 
by a path that remains in
$\UU(\delta,\bU)$.
\end{lemma}

\begin{proof}
We will assume that $S$ is connected
for simplicity, but the general case is handled easily one component at a time.

Notice that if $\gamma\in\mathcal{C}(S)$ and $X=\partial_cN_\nu\in\mathcal{T}(S)$, then
$l_\gamma(\nu)=l_\gamma(Q(X,\sigma(X)))$, since $Q(X,\sigma(X))$ is
the cover of $N_\nu$ associated to $\pi_1(S)$. Let $\delta_0>0$ be a lower bound
for $l_{\alpha_j}(Y)$ for all $\alpha_j\in\ba$ and $Y\in \sigma({\mathcal{T}(S)})$.
(The existence of $\delta_0$ follows from Thurston's Bounded Image Theorem.)
We may assume, without loss of generality, that $\delta<\delta_0$.

By Theorem \ref{length and projections},  we may choose
$K$ such that
$$
\mm_\gamma(X,\sigma(X)) > K \ \implies \ l_\gamma(\nu) < \delta
$$
and also suppose  that $K> 2/\delta$. 
There exists $R$ and $C$ such that \hbox{$\pi_{W}:\mathcal{T}(S)\to\mathcal{C}(W)$}
is coarsely $(R,C)$-Lipschitz
for all essential subsurfaces $W\subset S$
(see Lemma 2.3 in \cite{masur-minsky2}),
i.e. $d_{\mathcal{C}(W)}(X,Y) \le R d_{\mathcal{T}(S)}(X,Y)+C$
for all $X,Y\in\mathcal{T}(S)$. 
Let $h=h(K+RL+C,S)$ be the constant given by Lemma \ref{twist and shrink}, and
let $c=c(S)$ be the constant in part (3) of Lemma \ref{twist and shrink}. 

Lemma \ref{W connected} allows us to choose neighborhoods $W_i$
of $\lambda_i$ such that any two points in $\WW(\delta/2,\bW)$ are connected by a
path in $\WW(\delta/2,\bU)$. 

Choose $\ep$ small enough that (again by Theorem \ref{length and projections})
$$l_{\gamma}(Q(X,\sigma(X)))<\ep \ \implies \ 
\mm_{\gamma}(X,Y) > h,
$$
and choose neighborhoods $V_i$ of
$\lambda_i$ such that a 
$c$-neighborhood of $V_i$  in $\CC(S_i)$ is contained in $W_i$. 

If $\nu\in \UU(\ep,\bV)$ and $X=\partial_cN_\nu\in \mathcal{T}(S)$,
then Lemma \ref{twist and shrink} gives a path
$\{X_t:t\in[0,T]\}$ beginning at $X=X_0$, such that
\begin{enumerate}
\item
$l_{\alpha_j}(X_T) < 1/(K+RL+C) <\delta/2$ for all $\alpha_j\in\ba$, 
\item
$\mm_{\alpha_j}(X_t,\sigma(X)) > K+RL+C$ for all $\alpha_j\in\ba$ and all $t\in [0,T]$, and
\item
$\diam_{\CC(S_i)}(\pi_{S_i}(\{X_t\})) < c$.
\end{enumerate}
Let $\{\nu_t\ |\ t\in [0,T]\}$ be the associated path in $B$. Then, (1) and (3) imply that
$\nu_T\in  \WW(\delta/2,\bW)$. The facts that $\pi_W$ is coarsely $(R,C)$-Lipschitz
for all $W$, $d_{\mathcal{T}(S)}(\sigma(X),\sigma(X_t))\le L$,
$l_{\alpha_j}(\sigma(X_t))>\delta_0>\delta$ (for all $\alpha_j\in\ba$) and (2), imply that
$\mm_{\alpha_j}(X_t,\sigma(X_t)) > K$ for all $t$,
so $l_{\alpha_j}(\nu_t) < \delta$ for all $t$ and all $\alpha_j\in\ba$.
Combining this again with (3),
we see that the entire path
$\{ Q(X_t,Y)\}$ lies in $\UU(\delta,\bW)$.
 
We can now complete the argument exactly as in the proof of Lemma \ref{U to W}.
\end{proof}
\end{proof}

\section{Surface groups}
\label{general surface}

In this section we prove that quasifuchsian space doesn't self-bump at
quasiconformally rigid points in its boundary.
The proof is closely modeled on
the Bers slice case (\S\ref{bers slice}), with the main complication
being that we need to keep track of the ordering of the ends, and of
the relevant Margulis tubes, during the deformation. Theorem \ref{endlams}
allows us to keep track of the ordering of the ends, while Proposition \ref{hotsidehot} will
be used to control the ordering of the Margulis tubes.

\begin{theorem}{qfcase}
If $S$ is a closed surface and $\rho$ is a quasiconformally rigid
point in $\partial AH(S\times I)$, then there is no self-bumping at $\rho$.
\end{theorem}

Theorem \ref{noselfbump} follows from Theorems \ref{acylcase} and \ref{qfcase}.

\begin{proof}{}
We begin by constructing a neighborhood system for $\rho$ in
$QF(S)$. 
Let the upward-pointing end invariants  of $\rho$ be
denoted by a collection $\ba$ of simple closed curves on $S$ associated to 
upward-pointing cusps
and subsurfaces $\{S_i\}$ with laminations $\{\lambda_i\}$, and let its downward-pointing end invariants be denoted
by a collection $\bb$ of simple closed curves on $S$ associated to
downward-pointing cusps, and  subsurfaces $\{T_k\}$ with laminations
$\{\mu_k\}$. For all $i$ and $k$, let $U_i$ be a neighborhood of 
$\lambda_i\in\partial_\infty\CC(S_i)$ of $\CC(S_i)$ and let $V_k$ be a neighborhood of
$\mu_k$ in $\CC(T_k)$.
Let $\bU$ and $\bV$ denote the corresponding tuples of neighborhoods.
Define
$\UU(\delta,\bU,\bV)$ to be the set of all quasifuchsian groups $Q(X,Y)$
such that 
\begin{enumerate}
\item $\pi_{S_i}(X)\in U_i$   for all $i$,
\item
$l_{\alpha_j}(Q(X,Y)) < \delta$ for all $\alpha_j\in\ba$, 
\item
$\pi_{T_k}(Y)\in V_k $ for all $k$,
\item
$l_{\beta_l}(Q(X,Y)) < \delta$ for all $\beta_l\in\bb$, and
\item
if $\alpha_j\in\ba$ and $\beta_l\in\bb$ intersect on $S$, then
$\alpha_j$ lies above  $\beta_l$ in $Q(X,Y)$.
\end{enumerate}

\begin{lemma}{U is nbhd QF}
The sets $\UU(\delta,\bU,\bV)$
are the intersections with $QF(S)$ of a neighborhood system for $\rho$.
\end{lemma}

\begin{proof}
As in the proof of Lemma \ref{U is nbhd}, 
it suffices to show that a sequence $\{\rho_n = Q(X_n,Y_n)\}$ converges to
$\rho$ if and only if it is eventually contained in any
$\UU(\delta,\bU,\bV)$. 

Suppose $\{\rho_n=Q(X_n,Y_n)\}\subset QF(S)$ converges to $\rho$.
Then, by continuity of length, 
$\lim l_{\alpha_j}(Q(X_n,Y_n))=0$ for
all $\alpha_j\in\ba$ and $\lim l_{\beta_l}(Q(X_n,Y_n))=0$ for all $\beta_l\in\bb$.
Theorem \ref{endlams}
implies that $\{\pi_{S_i}(X_n)\}$ converges to $\lambda_i$ for all $i$
and $\{\pi_{T_k}(Y_n)\}$ converges to
$\mu_k$ for all $k$.  If $\alpha_j\in\ba$ and $\beta_l\in\bb$ intersect, then
Lemma \ref{hotsidehot} ensures that, for all large n, $\alpha_j$ lies above $\beta_l$
in $N_{\rho_n}$. Therefore, $\{\rho_n\}$ is eventually contained in any
$\UU(\delta,\bU,\bV)$. 

Now suppose that $\{\rho_n\}$ is eventually contained in any 
$\UU(\delta,\bU,\bV)$. We must first show that
any such $\{\rho_n\}$ has a convergent subsequence in $AH(S)$. 
If not, then some subsequence, still denoted $\{\rho_n\}$, converges
to a small action, by isometries, of $\pi_1(S)$ on an $\R$-tree $T$, i.e. there
exists $\{\epsilon_n\}$ converging to 0, such that $\{\epsilon_nl_\gamma(\rho_n)\}$ 
converges to the translation distance $l_T(\gamma)$  of the action of $\gamma$ on $T$
for any closed curve $\gamma$ on $S$ (see Morgan-Shalen \cite{morgan-shalen}).
Skora's theorem \cite{skora} implies that there exists a measured lamination $\nu$
on $S$ dual to the tree such
that $l_T(\gamma)=i(\nu,\gamma)$ for all $\gamma$.
If any $\alpha_j\in\ba$ or $\beta_l\in\bb$ intersects $\nu$, then we obtain an immediate
contradiction since $\lim l_{\alpha_j}(\rho_n)=0$ and $\lim l_{\beta_l}(\rho_n)=0$.
Therefore, $\nu$ must be contained both in some $S_i$ and in some $T_k$. The 
support of $\nu$ cannot agree with both $\lambda_i$ and $\mu_k$, since
$\lambda_i$ and $\mu_k$ do not agree, so $\nu$
must intersect either $\lambda_i$ or $\mu_k$ transversely.

Suppose without loss of generality that $\nu$ intersects $\lambda_i$ transversely. 
We will now show that the geodesics $[\pi_{S_i}(X_n),\pi_{S_i}(Y_n)]$
come uniformly close to a fixed point in $\CC(S_i)$ as
$n\to\infty$. (Recall that $[a,b]$ refers to {\em any} geodesic
connecting the points $a$ and $b$). Suppose first that some $\beta_l$
intersects $S_i$ essentially. Then since $l_{\beta_l}(\rho_n)$ is
bounded (in fact goes to 0), Theorem \ref{bounded curves near
  geodesic} gives a $D_0$ such that $\pi_{S_i}(\beta_l)$ stays within
$D_0$ of $[\pi_{S_i}(X_n),\pi_{S_i}(Y_n)]$.
Now suppose that $S_i$ is disjoint from $\bb$, and hence is contained
in $T_k$. Since $\mu_k$ fills $T_k$, it intersects $S_i$
essentially. Let $\tau_n$ be a shortest curve on $Y_n$ intersecting
$T_k$ essentially, such that $\pi_{T_k}(\tau_n) =
\pi_{T_k}(Y_n)$. Since $\pi_{T_k}(Y_n) \to \mu_k$, 
the Hausdorff limit of $\tau_n\intersect T_k$  must contain
$\mu_k$. Since $\mu_k$ intersects $S_i$, so must $\tau_n$ for high
enough $n$, and moreover eventually $d_{S_i}(\tau_n,\mu_k) \le 1$. 
Since $l_{\tau_n}(\rho_n)$ is bounded, Theorem \ref{bounded curves near geodesic} 
again tells us that $\pi_{S_i}(\tau_n)$, and hence the fixed point
$\pi_{S_i}(\mu_k)$, lie within bounded distance of
$[\pi_{S_i}(X_n),\pi_{S_i}(Y_n)]$.

Now, since $\pi_{S_i}(X_n)$ converges to
 $\lambda_i\in\partial_\infty\CC(S_i)$, we see that
$d_{S_i}(X_n,Y_n)\to\infty$.
Thus for large enough $n$ Theorem \ref{bounded curves near geodesic}
tells us that $\CC(S_i,\rho_n,L_0)$ is nonempty and within bounded
Hausdorff distance of $[\pi_{S_i}(X_n),\pi_{S_i}(Y_n)]$. In particular
there exists a sequence $\{\gamma_n\}\subset \CC(S_i)$ with
$\{l_{\gamma_n}(\rho_n)\}$ bounded, and $d_{S_i}(\gamma_n,X_n)$
bounded. The last bound implies that $\gamma_n \to \lambda_i$. 

However, the fact that $\lambda_i$ intersects $\nu$ essentially
implies, by Corollary 3.1.3 in Otal \cite{otal}, that $\lambda_i$ is
realizable in the tree $T$. Since $\gamma_n \to \lambda_i$, Theorem 4.0.1 in Otal \cite{otal} then
implies that $l_{\gamma_n}(\rho_n)\to\infty$, so we have achieved a
contradiction. We conclude that in fact $\{\rho_n\}$ has a convergent
subsequence. 

\medskip

Consider any accumulation point $\rho'$ of $\{\rho_n\}$. Each $\alpha_j\in\ba$
and $\beta_l\in\bb$ is associated to a cusp of $N_{\rho'}$. Proposition \ref{endlams}
implies that each $S_i$ is associated to an upward pointing geometrically infinite end
with ending lamination $\lambda_i$
and each $T_k$ is associated to a downward pointing end with ending lamination
$\mu_k$.
So, there exists a pared homotopy equivalence $h:(M_\rho,P_\rho)\to (M_{\rho'},P_{\rho'})$
which can be taken to be an orientation-preserving
homeomorphism on each $S_i$ and $T_k$.
Proposition 8.1 in Canary-Hersonsky \cite{canary-hersonsky} implies that there exists
a pared homeomorphism $h':(M_\rho,P_\rho)\to (M_{\rho'},P_{\rho'} )$ which agrees 
with $h$ on each $S_i$ and $T_j$. In particular, this implies that $\rho'$ is quasiconformally rigid. 

In order to apply the Ending Lamination Theorem it
remains to check that our pared homeomorphism $h'$ is orientation-preserving.
If $N_{\rho}$ has a geometrically infinite end associated to some $S_i$ or
$T_k$, then $h'$ is orientation-preserving on that surface, so it is
orientation-preserving. If $N_\rho$ has no geometrically infinite ends, then it
is a maximal cusp. So, each $\alpha_j\in\ba$ intersects some $\beta_l$. 
As $\rho'$ is quasiconformally rigid  and $\alpha_j$ lies above
$\beta_l$ in $N_{\rho_n}$ for all large enough $n$, Lemma
\ref{hotsidehot} implies that  $\alpha_j$ is
associated to an upward-pointing cusp of $N_{\rho'}$. Similarly, each $\beta_l\in\bb$
is associated to a downward-pointing cusp in $N_{\rho'}$,
so $h'$ must be orientation-preserving.
The Ending Lamination Theorem then allows us to conclude that $\rho'=\rho$. 
\end{proof}

{\bf Remark:} The convergence portion of the
above argument can also be derived from the main result of
Brock-Bromberg-Canary-Lecuire \cite{BBCL} or by using
efficiency of pleated surfaces as in Thurston's proof of the Double Limit Theorem
\cite{thurston2}.

\medskip

If $\delta>0$, $\bU$ and $\bV$ are as above,
then we define
$\WW(\delta,\bU,\bV)$ to be the set of all quasifuchsian groups $Q(X,Y)$
such that 
\begin{enumerate}
\item $\pi_{S_i}(X)\in U_i$   for all $i$,
\item
$l_{\alpha_j}(X) < \delta$ for all $\alpha_j\in\ba $, 
\item
$\pi_{T_k}(Y)\in V_k $ for all $k$, and
\item
$l_{\beta_l}(Y) < \delta$ for all $\beta_l\in\bb$.
\end{enumerate}

Lemma \ref{boundary order} and Bers' Lemma \ref{berslemma} give:

\begin{lemma}{WW subset UU}
If $\delta<\ep_0$, then
$\WW(\delta/2,\bU,\bV)\subset \UU(\delta,\bU,\bV)$.
\end{lemma}

Lemma \ref{boundary order} also allows us to restrict to neighborhoods
where the $\alpha_j\in\ba$ are not short on the bottom conformal boundary component
and the $\beta_l\in\bb$ are not short on the top conformal boundary component.

\begin{lemma}{hot cold}
There exist neighborhoods $(U_i)_0$ of $\lambda_i$ in $\CC(S_i)$ and
$(V_k)_0$ of $\mu_k$ in $\CC(T_i)$ such that if 
$Q(X,Y)\in \UU(\ep_0,\bU_0,\bV_0)$, then
$l_{\beta_l}(X) > \ep_0$ and $l_{\alpha_j}(Y) > \ep_0$ for all $\beta_l\in\bb$ and 
$\alpha_j\in\ba$.
\end{lemma}

\begin{proof}{}
Suppose that $l_{\beta_l}(X)\le \ep_0$ for some $\beta_l\in\bb$. If $\beta_l$ intersects some
$\alpha_j\in\ba$, then Lemma \ref{boundary order} would imply that
$\beta_l$ lies above $\alpha_j$ which is a contradiction.
If $\beta_l$ does not intersect any $\alpha_j$, then it lies in some $S_i$.
Then $d_{S_i}(X,\beta_l)\le 2$.
So, if we choose the neighborhood $(U_i)_0$ to
have the property that $\pi_{S_i}(\beta_l)$ does not lie in the 2-neighborhood of
$(U_i)_0$,  we again have a contradiction. 

The proof that $l_{\alpha_j}(Y) > \ep_0$ for all $\alpha_j\in\ba$ is similar.
\end{proof}

Theorem \ref{qfcase} now follows from:

\begin{lemma}{U to W QF}
Given $\delta>0$ and neighborhoods $U_i$ of $\lambda_i$ and $V_i$ of
$\mu_i$, there exists $\ep>0$ and neighborhoods $U''_i$ of $\lambda_i$ in $\CC(S_i)$
and $V''_k$ of $\mu_k$ in $\CC(T_k)$ such that any two 
points in $\UU(\ep,\bU'',\bV'')$ can be connected 
by a path that remains in
$\UU(\delta,\bU,\bV)$.
\end{lemma}

\begin{proof}
Without loss of generality, we may assume $\delta<\delta_0$ 
(from Lemma \ref{alpha below Z})
and $\bU\subset \bU_0$, $\bV\subset \bV_0$ (from Lemma
\ref{hot cold}). By Theorem \ref{bounded curves near geodesic}, we may further
assume that if $W$ is an 
essential subsurface of $S$, $\gamma\in\CC(S,W)$ and
$l_\gamma(Q(X,Y))<\delta$,
then $\pi_W(\gamma)$ lies within $D_0$ of any geodesic joining $\pi_W(X)$ to
$\pi_W(Y)$.

By Theorem \ref{length and projections}, we may choose
$K$ such that
$$
\mm_\gamma(X,Y) > K \ \implies \ l_\gamma(Q(X,Y)) < \delta
$$
and also suppose  $K> 2/\delta$. 
Let $h=h(K,S)$ be the constant given by Lemma \ref{twist and shrink}, and
let $c=c(S)$ be the constant in part (3) of Lemma
\ref{twist and  shrink}. Let $d_0$ be the constant from Lemma \ref{alpha below Z}.

Lemma \ref{W connected} implies that we may choose
neighborhoods $U_i'$ of $\lambda_i$ in $\CC(S_i)$  and  neighborhoods $V_k'$
of $\mu_k$ in $\CC(T_k)$ such that  any two points in $\WW(\delta/2,\bU',\bV')$ 
are connected by a path in $\WW(\delta/2,\bU,\bV)$. Moreover,
we may further assume that if $\beta_l\in\bb$ is contained in $S_i$, then
$$d_{S_i}(\beta_l,\gamma)>R=m_1+D_0+1$$ 
for all $\gamma\in U_i'$.

Choose $\ep$ small enough that (again by Theorem \ref{length and projections})
$$l_{\gamma}(Q(X,Y))<\ep \ \implies \ 
\mm_{\gamma}(X,Y) > h'=h+2d_0 +D_0+m_1+ c.
$$
Finally, choose neighborhoods $U''_i$ of
$\lambda_i$  in $\CC(S_i)$ and $V''_k$ of $\mu_k$ in $\CC(T_k)$ such that a 
$c$-neighborhood of $U''_i$  in $\CC(S_i)$ is contained in $U'_i$, and
a $c$-neighborhood of $V''_k$ in $\CC(T_k)$ is contained in $V'_k$. 

Suppose that $Q(X,Y)\in\UU(\ep,\bU'',\bV'')$. 
Then, by our choice of $\ep$, \hbox{$\mm_{\alpha_j}(X,Y) > h$} for each
$\alpha_j\in\ba$, and 
so Lemma \ref{twist and shrink} can be applied to give
a path \hbox{$\{X_t\ |\ t\in[0,T]\}$} beginning at $X=X_0$ such that 
\begin{enumerate}
\item
$l_{\alpha_j}(X_T) < 1/K <\delta/2$ for all $\alpha_j\in\ba$,
\item
$\mm_{\alpha_j}(X_t,Y) > K$ for all $\alpha_j\in\ba$ and $t\in [0,T]$, and
\item
$\diam_{\CC(S_i)}(\{ \pi_{S_i}(X_t)\ |\ t\in [0,T]\}) < c$ for all $i$.
\end{enumerate}
Condition (2), Bers' Lemma \ref{berslemma} and our choice of $K$,
give that \hbox{$l_{\alpha_j}(Q(X_t,Y))<\delta$} for all $\alpha_j\in\ba$ and all $t\in [0,T]$. 
Condition (3) and our choice of  $U_i''$ give that 
$\pi_{S_i}(X_t) \in U'_i$ for all $i$ and all $t\in [0,T]$.

In order to
guarantee that $Q(X_t,Y)\in \UU(\delta,\bU',\bV')$ for all $t$, it remains to check that
$l_{\beta_l} (Q(X_t,Y))<\delta$ for all $\beta_l\in\bb$ and that
each $\beta_l$ remains correctly ordered
with respect to relevant $\alpha_j\in\ba$. Recall that, again by our choice of $\ep$,
$$\mm_{\beta_l}(X,Y) > h'$$
for all $\beta_l\in\bb$.
We will additionally need to establish that
\begin{equation} \label{big m}
\mm_{\beta_l}(X_T,Y) > h.
\end{equation}
Condition \ref{big m} is necessary to invoke Lemma \ref{twist and shrink} to
construct the deformation of the bottom conformal structure $Y$.

If $l_{\beta_l}(Y)\le 1/h'<\delta/2$ then
\hbox{$\mm_{\beta_l}(X_t,Y)>h$} and, by Bers' Lemma \ref{berslemma},
$l_{\beta_l}(Q(X_t,Y))<\delta$ for all $t$. Lemma \ref{boundary order} then implies that if
$\beta_l$ intersects $\alpha_j\in\ba$ in $S$, then 
$\beta_l$ lies below $\alpha_j$ in $Q(X_t,Y)$ for all $t$. 

If $l_{\beta_l}(Y)> 1/h'$, then, since $l_{\beta_l}(X)>\ep_0$ (by Lemma \ref{hot cold})
and \hbox{$\mm_{\beta_l}(X,Y)\ge h'$,} there must be a 
subsurface $Z_l$ with
$\beta_l\subset \boundary Z_l$
such that $$d_{Z_l}(X,Y) > h'.$$
If $Z_l$ does not intersect $\ba$, then, by Lemma \ref{twist  and shrink}(3), 
$${\rm diam}_{\CC(Z_l)}(\{\pi_{Z_l}(X_t)\ |\ t\in [0,T]\}) < c,$$
so $d_{Z_l}(X_t,Y) > h'-c>h $  for all $t$ and $\beta_l$ does not intersect 
$\ba$. Therefore, $l_{\beta_l} (Q(X_t,Y))<\delta$  for all $t$
and condition (\ref{big m}) holds.

Suppose $\beta_l$ intersects $\alpha_j\in\ba$ on $S$.
For each $t\le T$, we know that \hbox{$l_{\alpha_j}(Q(X_t,Y))<\delta_0$.} 
Lemma \ref{alpha below Z} asserts that if $d_{Z_l}(X_t,\alpha_j) \ge d_0$, then
$\beta_l$ lies above $\alpha_j$ in $Q(X_t,Y)$.
Since $\beta_l$ lies below $\alpha_j$ in $Q(X,Y)$, we have that
\hbox{$d_{Z_l}(X,\alpha_k) <d_0$,}
so  $$d_{Z_l}(\alpha_j,Y)\ge d_{Z_l}(X,Y)-d_{Z_l}(X,\alpha_j)>h'-d_0>h+d_0.$$
It then follows, again from Lemma \ref{alpha below Z} (this time with the roles
of $X$ and $Y$ reversed), that $\alpha_j$ lies above $\beta_l$
in $Q(X_t,Y)$ for all $t$. So, one must have $d_{Z_l}(X_t,\alpha_j)<d_0$ for all $t$
and hence $$d_{Z_l}(X_t,Y)\ge d_{Z_l}(\alpha_j,Y)- d_{Z_l}(X_t,\alpha_j)>h$$
which in turn implies that $l_{\beta_j}(Q(X_t,Y))<\delta$ for all $t$.  In particular,
we have established  condition (\ref{big m}).

It remains to consider the case where $Z_l$ intersects some $\alpha_j\in\ba$, but
$\beta_l$ does not intersect  $\ba$. In this case, we do not
need to worry about the ordering of $\beta_l$, but only need to
check that $l_\beta(Q(X_t,Y))<\delta$ for all $t\in [0,T]$ and verify condition
(\ref{big m}).
Notice that $\beta_l$ is contained
in some $S_i$. We see that $d_{S_i}(\beta_l,X)>R$, since $\pi_{S_i}(X)\in U_i'$.
Since $l_{\beta_l}(Q(X,Y))<\delta$, $\beta_l$ lies within $D_0$ of any
geodesic joining $\pi_{S_i}(X)$ to $\pi_{S_i}(Y)$. Therefore,
$$d_{S_i}(X,Y)\ge R-D_0>m_1.$$
Since $d_{Z_l}(X,Y)>h'>m_1$,
Lemma \ref{partialorder} implies that $S_i$ and $Z_l$ are \hbox{$\prec$-ordered}
in $\LL_b(X,Y)$ where $b=\min\{R-D_0,h'\}>m_1$.
Since $$d_{S_i}(X,\partial Z_l)\ge d_{S_i}(X,\beta_l)-1>R-1>m_1,$$
Lemma \ref{partialorder}(3) implies that $Z_l\prec S_i$. Therefore,
Lemma \ref{partialorder}(2) shows that  $d_{Z_l}(\partial S_i,X)\le m_1$,
which implies that $d_{Z_l}(\partial S_i,Y)\ge h'-m_1$.
But, since $l_{\alpha_j}(Q(X_t,Y))<\delta$  if $\alpha_j$ is a component of $\partial S_i$, we conclude, as above,
that $$d_{Z_l}(X_t,Y)\ge d_{Z_l}(\partial S_i,Y) -D_0\ge h'-m_1-D_0>h$$ for all $t\in [0,T]$. Therefore, 
$l_{\beta_l}(Q(X_t,Y))<\delta$ for all $t$ and condition (\ref{big m}) holds.

We have considered all cases, so have completed the proof that
$Q(X_t,Y)\in \UU(\delta,\bU',\bV')$ for all $t\in [0,T]$.

Now we can fix $X_T$ and apply Lemma \ref{twist and shrink} to the bottom
side, obtaining a path $\{Y_t\ |\ t\in[0,T']\}$ beginning at $Y=Y_0$ such that
\begin{enumerate}
\item $l_{\beta_l}(Q(X_T,Y_t)) < \delta$ for all $t\le T'$ and $\beta_l\in\bb$,
\item $\pi_{T_k}(Y_t) \in V'_k$ for all $k$ and  $t\le T'$, and
\item $l_{\beta_l}(Y_{T'}) < \delta/2$ for all $\beta_l\in\bb$.
\end{enumerate}
Recall that $\pi_{S_i}(X_T)\in U_i$ for all $i$  and $l_{X_T}(\alpha_j)<\frac{\delta}{2}<\ep_0$ 
for all $\alpha_j\in\ba$. Therefore, Lemma \ref{boundary order} implies that $\alpha_j$
lies above $\beta_l$ in $Q(X_T,Y_t)$ for all $t\in [0,T']$ whenever
$\alpha_j$ and $\beta_l$ intersect on $S$. Therefore, the path
$\{ Q(X_T,Y_t)\ |\ t\in [0,T']\}$ lies entirely in $\UU(\delta,\bU',\bV')$.
The concatenation of the paths  $\{ Q(X_t,Y)\ |\ t\in [0,T]\}$ and
$\{ Q(X_T,Y_t)\ |\ t\in [0,T']\}$  remains in
$\UU(\delta,\bU',\bV')$,  and joins $Q(X,Y)$ to a  point
$Q(X_T,Y_{T'})\in\WW(\delta/2,\bU',\bV')$. 

Now since any two points in
$\WW(\delta/2,\bU',\bV')$ can be connected by a path in 
$\WW(\delta/2,\bU,\bV)$, and since 
$\WW(\delta/2,\bU,\bV) \subset\UU(\delta,\bU,\bV)$,
by Lemma \ref{WW subset UU},
we conclude that any two points in $ \UU(\ep,\bU'',\bV'')$ can be connected
by a path in $\UU(\delta,\bU,\bV)$. 
\end{proof}

\end{proof}

\end{document}